\documentclass[11pt]{article}
\usepackage[utf8]{inputenc}

\usepackage{amsmath}
\usepackage{amsfonts}
\usepackage{accents}
\usepackage{graphicx}
\usepackage{geometry}
\usepackage{cancel}
\usepackage{fancyhdr}
\usepackage{dsfont}
\usepackage{mathrsfs}
\usepackage{upgreek}
\usepackage{amsthm}
\usepackage{amssymb}
\usepackage{mathtools}
\usepackage{pdfpages}
\usepackage{titlesec}
\usepackage{hyperref}
\usepackage{nicematrix}
\usepackage[sort&compress,comma,numbers]{natbib}
\setcitestyle{square}
\usepackage{newtxtext,newtxmath}
\usepackage{chngcntr}
\usepackage{multirow}
\usepackage{booktabs}
\usepackage{enumitem}

\newtheorem{theorem}{Theorem}[section]

\newtheorem{proposition}{Proposition}[subsection]

\newtheorem{definition}{Definition}[subsection]
\newtheorem{assumption}{Assumption}[subsection]
\newtheorem{corollary}{Corollary}[subsection]

\DeclareMathOperator*{\argmin}{argmin}

\renewenvironment{abstract}
 {\par\noindent\textbf{\abstractname.} \ignorespaces}
 {\par\medskip}
 
 \renewcommand\thesubsubsection{\arabic{subsubsection}.}
 \titleformat{\subsubsection}{\itshape}{\thesubsubsection}{1em}{\itshape}

\title{\textbf{Real-time inverse solutions via neural matrix operators}}

\author{Julie V. Pham\textsuperscript{1,}\footnote{Graduate Research Assistant, Department of Aerospace Engineering and Engineering Mechanics.} , \hspace{1pt} Thomas O'Leary-Roseberry\textsuperscript{2,}\footnote{Assistant Professor, Department of Mathematics.} ,  \hspace{1pt} Omar Ghattas\textsuperscript{1,}\footnote{Professor, Walker Department of Mechanical Engineering, Oden Institute for Computational Engineering and Sciences.} , \hspace{1pt} and Karen E. Willcox\textsuperscript{1,}\footnote{Professor, Oden Institute for Computational Engineering and Sciences. External Faculty, Santa Fe Institute.}}
\date{\emph{\textsuperscript{1}The University of Texas at Austin, Austin, TX 78712, USA}\\
\textsuperscript{2}\emph{The Ohio State University, Columbus, OH 43210, USA}\\ \
}

\begin{document}

\maketitle

\begin{abstract}

Rapid data assimilation is required for real-time prediction and control in digital twins. For many physical systems, the data assimilation task requires the solution of a physics-constrained inverse problem, which is often computationally intractable in real time using traditional physics solvers. This work presents a reduced-basis neural operator approach to enable real-time inverse problem solutions in the digital twin setting.
Our approach specifically targets the large class of problems with spatiotemporal dynamics governed by partial differential equations (PDEs) that are parameterized nonlinearly with respect to model parameters $m$, and linearly with respect to inversion parameters $q$.
Based on this physical structure, our neural operator approximates the nonlinear map from the model parameters $m$ to the parameter-to-observable operator $\mathcal{F}(m)$ in a reduced subspace. 
Since the output of the neural operator is the parameter-to-observable operator itself (manifested as a matrix), we refer to this approach as NEural Matrix Operator (NEMO). 
With NEMO, for new given $m$, we enable a closed-form inverse problem solution for $q$ in a reduced subspace. We apply NEMO in two real-world applications: contaminant transport initial condition identification, and hypersonic vehicle load identification. We show that NEMO delivers high quality inverse problem solutions for data assimilation in real time, with over three orders of magnitude speedup compared to constructing the reduced operator with the PDE solver. Further, NEMO demonstrates comparable inverse performance to a state-of-the-art multiple-input neural operator, while reducing online computational complexity by over an order of magnitude and providing real-time uncertainty quantification.

\end{abstract}

\section{Introduction}

Our goal is the rapid real-time solution of inverse problems arising from physical systems with spatiotemporal dynamics governed by partial differential equations (PDEs) with an unknown source, such as a forcing term, initial condition, or boundary condition. Such problems arise in a wide variety of physical systems, where uncertain environments may impose an unknown source that cannot be measured directly. 
In particular, such problems underpin the deployment of digital twins in many applications across science, engineering, and medicine, where solving the inverse problem under extreme computational cost and time constraints is a critical challenge~\cite{NAP26894}.
As new data are acquired, an inverse problem must be solved under real-time operating constraints to update the digital twin and thus enable the downstream tasks of prediction, planning, and control. In these settings, using traditional physics solvers is infeasible due to their high computational cost. 
To address this challenge, we aim to solve inverse problems for PDE-governed systems in real time using a neural operator approach. Specifically, we target the class of PDEs that are parameterized nonlinearly with respect to model parameters $m$, and linearly with respect to inversion parameters $q$, which represent the unknown source. These parameters are related to the observational data $\mathbf{d}$ through the parameter-to-observable (p2o) operator $\mathcal{F}(m)$, which encompasses the governing physics of the PDE, and is nonlinearly parameterized by the model parameters $m$. 
Problems of this structure arise in many practical settings, including reconstruction of earthquake ruptures from seismometer waveforms for uncertain Earth model wave speeds, inversion of greenhouse gas surface fluxes from satellite measurements for uncertain atmospheric flow, identification of underwater pipeline leaks from acoustic pressure measurements for uncertain ocean wave speed, and inference of underground contaminants from well observations for uncertain subsurface flows, to name a few.

We propose learning the operator $\mathcal{F}(m)$ via a reduced-basis neural operator methodology as depicted in Figure~\ref{fig:method-outline}. Here, the learning task is to approximate the map $m \mapsto \mathcal{F}(m)$, which outputs the operator itself, as opposed to learning the action of $\mathcal{F}$ on the parameters $q$ or $m$ in typical neural operator approaches. We refer to this surrogate map as NEural Matrix Operator (NEMO). The reduced basis structure enables efficient and scalable learning of the map. It also naturally leads to a fast-to-solve inverse problem due to the availability of a closed-form least-squares solution. This solution can be computed with low costs and
$\mathcal{O}(dr^2 + r^3)$ complexity, where $d$ is the dimension of the observables, and $r$ is the dimension of the reduced basis. Since $r$ can be made small for many problems, this translates into an algorithm that is suitable for edge and real-time computing settings.

\begin{figure}[b!]
    \centering
    \includegraphics[width=0.98\linewidth]{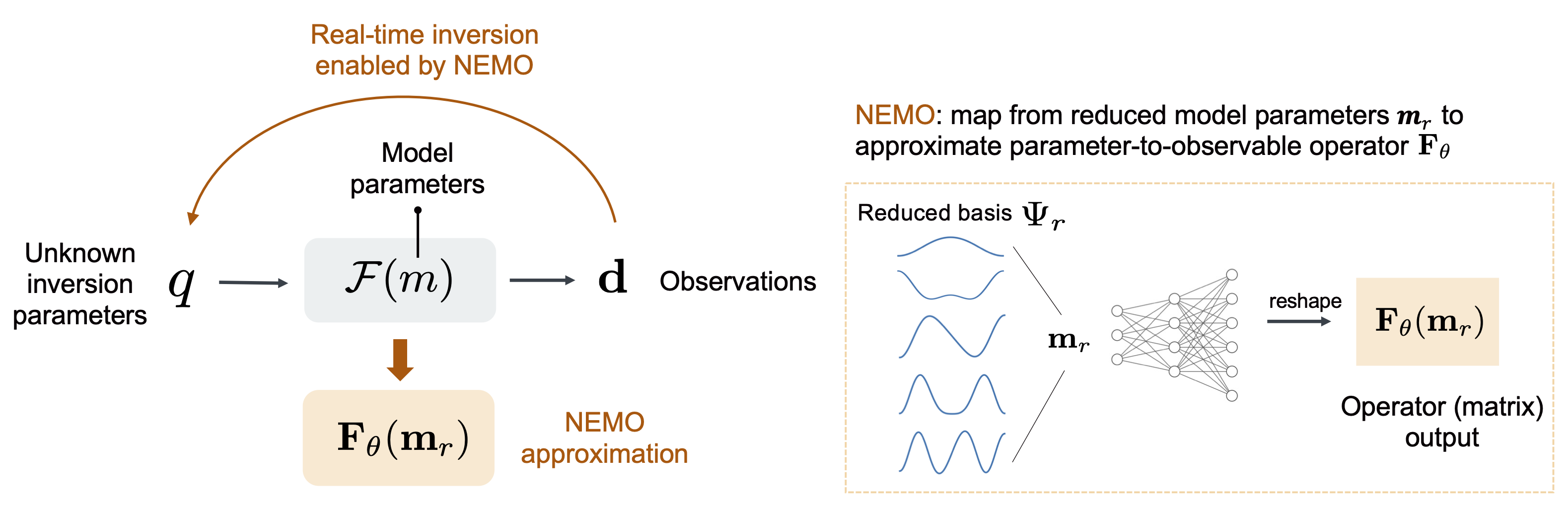}
    \caption{\textbf{NEural Matrix Operator (NEMO) framework.} NEMO provides a surrogate $\mathbf{F}_\theta(\mathbf{m}_r)$ for the parameter-to-observable operator $\mathcal{F}(m)$ by learning the map from a reduced basis representation of the model parameters to a reduced basis representation of the operator $\mathcal{F}(m)$. Since this preserves the physics structure of the problem, we enable inversion for the parameters $q$ in real time for fast data assimilation in digital twins. }
    \label{fig:method-outline}
\end{figure}

Broadly, neural operator surrogates are a promising direction for real-time, edge computing settings due to their rapid evaluation speed, and readily available gradients. In the many-query setting, neural operators can function as a surrogate model for an expensive PDE evaluation, enjoying fast evaluation speed at the expense of an initial training cost. Several prior works have developed methods for neural operators \cite{li2021, lu2021, bhattacharya2021, olearyroseberry2023dino} and demonstrated their potential in inverse problems~\cite{cao2025derivative, cao2024lazydino, wang2024, dai2023, jiao2024solvingforwardinversepde,lu2022multifidelity,kaltenbach2023semi}.
However, these methods primarily focus on offline inverse problems, and require solving a (typically nonconvex) nonlinear optimization problem governed by the neural operator as well as expensive sampling for every new observational data instance, for which convergence to a solution remains a computational bottleneck. Some neural operator approaches learn the inverse map directly \cite{molinaro2023_NIO, cho2025physicsinformed}. However, these approaches do not exploit the physical structure of the target problem and instead build a generic, hard-to-interpret nonlinear approximation.
In this work, we preserve the physical structure and consequently the availability of a closed-form solution for rapid, real-time evaluation. We demonstrate how NEMO can be used to solve inverse problems in two applications: contaminant initial condition inversion for emergency response planning, and hypersonic aerodynamic pressure estimation for guidance and control. We summarize our contributions as follows: 
\begin{enumerate}
    \item \textbf{Real-time methodology for solving inverse problems} for PDE-governed systems with unknown source parameters via nonlinearly parameterized least-squares solutions. 
    \item A novel \textbf{reduced-basis neural operator} that outputs the reduced parameter-to-observable operator, enabling the rapid solution of the aforementioned inverse problems. 
    \item \textbf{Supporting theory} establishing the consistency of our approach, and discretization dimension independence in the case of smooth maps. 

\end{enumerate}
The remainder of the paper is outlined as follows. Section~\ref{sec:methods} details the NEMO methodology for solving inverse problems, and Section~\ref{sec:results} demonstrates the NEMO performance on the application problems. Section~\ref{sec:conclusion} concludes and discusses directions for future work.

\section{Methods}\label{sec:methods}

This section presents the methodology for real-time inverse solutions using NEMO, where Section~\ref{sec:ip-formulation} defines the inverse problem formulation, Section~\ref{section:finding-nemo} introduces the NEMO method, and Section~\ref{section:ip_subspace} presents the subspace inverse problem and supporting theory.

\subsection{Problem Definition}\label{sec:ip-formulation}
We define an inverse problem of the form 
\begin{align} \label{eq:ip_system}
    \min_q \frac{1}{2}\|\mathcal{F}(m)q - \mathbf{d}\|_{\Gamma^{-1}}^2 + \frac{\gamma}{2}\|q\|^2_\mathcal{Q}
\end{align}
with the state equation and parameter-to-observable operator of the governing PDE given by:
\begin{subequations} \label{eq:fwd_map_system}
\begin{align}
    &A(m)u = Cq \quad &\text{(State equation / PDE)} \label{eq:state_eq}\\
    &\mathcal{F}(m) = B A^{-1}(m)C \quad &\text{Parameter-to-observable (p2o) operator}
\end{align}
\end{subequations}
Here, $A(m):\mathcal{U} \rightarrow \mathcal{U}'$ is a parametrized PDE operator, $u\in \mathcal{U}$ is the state, and $q \in \mathcal{Q}$ are the unknown inversion parameters (e.g., forces, boundary conditions, initial conditions) with coupling operator $C:\mathcal{Q}\rightarrow \mathcal{U}'$, where $\mathcal{U}'$ is the topological dual of $\mathcal{U}$.  This coupling operator allows for $q$ to abstractly represent sources, boundary conditions and initial conditions; for a more in-depth discussion, see Appendix~\ref{appendix:coupling}. The model parameters $m \in \mathcal{M}$ can be infinite-dimensional or finite-dimensional. We assume that $m$ are estimated or directly observed by means external to the present inverse problem. We additionally assume that $\mathcal{U},\mathcal{Q},\mathcal{M}$ are real-valued separable Hilbert spaces, where associated norms are introduced via the inner products e.g., $\|q\|^2_\mathcal{Q} = \langle q , q \rangle_\mathcal{Q}$. The observational data $\mathbf{d} \in \mathbb{R}^{d}$ are assumed to be finite-dimensional for simplicity, however, the present methodology can be extended to full field observations using similar dimension reduction strategies. The PDE observables corresponding to the data, $\mathcal{F}(m)q$, are obtained via the observation operator $B\in \mathcal{L}(\mathcal{U},\mathbb{R}^d)$, where $\mathcal{L}(\mathcal{U},\mathbb{R}^d)$ is the space of linear operators from $\mathcal{U}$ to $\mathbb{R}^{d}$; similarly, the parameter-to-observable operator $\mathcal{F}(m) \in \mathcal{L}(\mathcal{Q},\mathbb{R}^d)$.  We consider a zero mean, additive Gaussian noise model, such that
\begin{equation}
    \mathbf{d} = \mathcal{F}(m)q + \epsilon
\end{equation}
where $\epsilon \sim \mathcal{N}(0,\Gamma)$, and $\Gamma$ is symmetric positive definite. 

The corresponding solution of the inverse problem \eqref{eq:ip_system} is given by the least squares solution

\begin{equation}\label{eq:ip-solution}
    q^\star = {\big[\underbrace{\mathcal{F}(m)^*\Gamma^{-1}\mathcal{F}(m)}_{\mathcal{H}(m)} + \gamma I_\mathcal{Q}\big]}^{-1}\mathcal{F}(m)^*\Gamma^{-1}\mathbf{d},
\end{equation}
where 
$^*$ denotes the adjoint, (e.g., $y^T\mathcal{F}(m)q = \langle y, \mathcal{F}(m)q\rangle_2 = \langle \mathcal{F}(m)^*y,q\rangle_\mathcal{Q}$), $\gamma$ is the regularization parameter, and $I_\mathcal{Q}$ is the identity operator on $\mathcal{Q}$. 
Computing this solution using traditional PDE approaches for real-time inversion is infeasible, since each application of $\mathcal{F}$ or its adjoint require the potentially expensive application of a PDE solution operator ($A^{-1}$). We are additionally concerned with the dimensionality of numerical approximations of the space $\mathcal{Q}$, since the inverse solution requires solution of linear systems with  operator given by $\mathcal{H}(m) + \gamma I_\mathcal{Q}$. Equation~\ref{eq:ip-solution} shows that the inverse solution depends on $m$ via $\mathcal{F}(m)$. Since we need to rapidly and repeatedly solve \eqref{eq:ip_system} for varying $m$ and $\mathbf{d}$, we seek a surrogate for the nonlinearly parametrized linear operator $\mathcal{F}$.
In Section \ref{section:finding-nemo},
we propose a reduced-basis neural operator method to learn the map $m \mapsto \mathcal{F}(m)$. 
Learning this map in a reduced subspace leads to a subspace definition of the inverse problem, with solutions of the same form as~(\ref{eq:ip-solution}). As we will see in Section \ref{section:ip_subspace}, the subspace structure enables fast, real-time inversion.

\subsection{Finding NEMO: Neural Matrix Operator Learning}
\label{section:finding-nemo}

We wish to learn an approximation of the map $\mathcal{F}:\mathcal{M} \rightarrow \mathcal{L}(\mathcal{Q},\mathbb{R}^d)$ in order to enable rapid inversion (Eq.~\ref{eq:ip-solution}). 
We employ reduced-basis neural operator learning approaches to learn matrix (linear operator) valued outputs instead of  vector (function) valued outputs that are typical of these methods.
By constructing a linear operator output directly, we ensure that our representation respects the known linear relationship between $\mathcal{Q}$ and the observable space $\mathbb{R}^d$.
We utilize reduced basis representations for both the functions $q$ and $m$, 
since the operator $\mathcal{F}(m)$ admits a low rank representation for a wide spectrum of elliptic, parabolic, and even hyperbolic PDE problems \cite{ghattas_willcox_2021}. This stems from the smoothing and compact nature of many PDE solution operators, as well as the finite dimensionality of observation space $\mathbb{R}^d$. While we use linear subspace representations of $q$ and $m$, we note that our framework extends naturally to nonlinear dimension reduction for 
$\mathcal{M}$. 

Specifically, we expand the functions $q$ and $m$ using truncated expansions in orthonormal bases (ONB), as is done in PCANet \cite{bhattacharya2021,hesthaven2018non} and other reduced basis neural operator methods \cite{o2022derivative,luo2025dimension,herrmann2024neural}. Given ONBs $\{\phi_i\}_{i}$ and $\{\psi_j\}_{j}$ for $\mathcal{Q}$ and $\mathcal{M}$ respectively, we represent $q$ and $m$ as truncated expansions
\begin{equation} \label{eq:qr_mr_expansion}
    q_r = \sum_{i=1}^{r_q} \langle q,\phi_i\rangle_\mathcal{Q}\phi_i = \Phi_r\Phi_r^*q \qquad m_r = \sum_{j=1}^{r_m} \langle m , \psi_j\rangle_\mathcal{M}\psi_j = \Psi_r\Psi_r^*m
\end{equation}
where $\Phi_r \in \mathcal{L}(\mathbb{R}^r,\mathcal{Q})$ and $\Psi_r \in \mathcal{L}(\mathbb{R}^r,\mathcal{M})$ are basis expansion operators. Their adjoints inherit the scaling (Riesz maps) for the associated inner products.
We can now represent the functions in their truncated coefficient representations, $\mathbf{q}_r = \Phi_r^*q \in \mathbb{R}^{r_q}$, and $\mathbf{m}_r = \Psi_r^*m \in \mathbb{R}^{r_m}$. We construct a neural network surrogate for the coefficient mapping, $\mathbf{F}_\theta : \mathbb{R}^{r_m} \rightarrow \mathbb{R}^{d\times r_q}$, parameterized by neural network weights $\theta$. The neural network weights are found via an empirical risk approximation of the following problem
\begin{equation}
    \min_\theta \mathbb{E}_{m \sim \mu} \left[\|\mathbf{F}_\theta(\Psi_r^*m) - \mathcal{F}(m)\Phi_r\|^2_F\right],
\end{equation}
where $\mu$ is a distribution of interest over $\mathcal{M}$. Because we use a truncated ONB expansion of $\mathcal{Q}$ to reduce the dimension of  $\mathcal{F}$, the output will always be a truncated coefficient expansion of the operator (i.e., with entries $e_i^T\mathcal{F}(\cdot)\phi_j$). This is always a matrix, and thus we refer to the approach as neural matrix operator (NEMO).  Note that the output of the neural network can in practice be the flattened matrix, which is then reshaped as shown in Figure~\ref{fig:method-outline}. Training data for this learning task are computed by drawing independent samples of the parameter $m_i \sim \mu$, and computing the corresponding reduced solution operator $\mathcal{F}(m_i)\Phi_r$. 

\paragraph{Choices of bases}

Reduced basis approximation is a powerful and well-studied dimension reduction tool in computational science \cite{quarteroni2015reduced, benner2017model}. By finding appropriate reduced bases for $m$ and $q$ and representing them by their reduced coefficients, we can simultaneously reformulate the operator learning problem and inverse problem so that their computational complexities do not depend on ambient discretization dimensions. This ensures that the derived algorithms will be scalable, and also has additional benefits in the inverse problem, such as serving as effective regularization \cite{vogel2002computational,benning2018modern}.
We consider a few relevant choices. First, we consider the proper orthogonal decomposition (POD) \cite{quarteroni2015reduced},
\begin{equation}
    \mathbb{E}_{m \sim\mu} [\langle m-\overline{m},\psi_j\rangle_\mathcal{M}(m-\overline{m})] = \lambda_j^{(m)}\psi_j \qquad \langle \psi_i,\psi_j\rangle_{\mathcal{M}} = \delta_{ij}.
\end{equation}
We note that this is equivalent to PCA in function space \cite{ramsay2002applied} and is related to the Karhunen--Lo\`eve expansion of the function $m$. Alternatively, for a goal-oriented dimension reduction strategy for the domain $\mathcal{Q}$ of the operator $\mathcal{F}(m)$, we can compute the following: 
\begin{equation} \label{eq:goal_oriented_F_reduction}
   \underbrace{ \mathbb{E}_{m \sim \mu} \left[ \mathcal{F}(m)^*\Gamma^{-1}\mathcal{F}(m) \right]}_{\overline{\mathcal{H}}} \phi_i = \lambda^{(\overline{\mathcal{H}})}_i \phi_i \qquad \langle \phi_i,\phi_j\rangle_\mathcal{Q} = \delta_{ij}.
\end{equation}
With this approach, as a direct corollary of Fan's Theorem \cite{fan1949theorem} as stated in \cite{bhattacharya2021}, we can derive the following error bound:
\begin{equation} \label{eq:f_fans_thm}
    \mathbb{E}_{m \sim \mu} \left[\| \Gamma^{-\frac{1}{2}}\mathcal{F}(m)(I_\mathcal{Q} - \Phi_r\Phi_r^*)\|_{HS(\mathcal{Q},\mathbb{R}^d)}^2 \right] = \sum_{j=r_q+1}^\infty \lambda_j^{(\overline{\mathcal{H}})}.
\end{equation}
This bound demonstrates that $\mathcal{F}$ is well approximated when reduced by $\Phi_r$ on the right when the eigenvalues of $\overline{\mathcal{H}}$ decay rapidly, which we expect due to the finite dimensionality of $\mathbb{R}^d$ and the smoothing nature of many PDE solution operators. 
See Appendix~\ref{appendix:dim_reduction} for a more thorough discussion, including errors due to the empirical estimation of this basis from samples. 

\paragraph{Training data generation} 

To generate training data, we obtain independent samples of the parameter $m_i \sim \mu$ and compute the corresponding reduced solution operator $\mathcal{F}(m_i)\Phi_r$. This process can be made efficient when amortizing linear algebra associated with the PDE solution operator $A^{-1}(m)$. When utilizing sparse factors, the costs of factorization can be amortized over the application of triangular factors to the columns of $\Phi_r$. Alternatively, when utilizing iterative methods, the need for many right-hand side evaluations incentivizes the investment in advanced preconditioning methods. Further training data generation details for each numerical example can be found in Appendix~\ref{appendix:numerics}. 

\subsection{Subspace Inverse Problem} \label{section:ip_subspace} 

By restricting the inverse problem to the span of $\Phi_r$, we recover the following subspace normal equations:
\begin{equation}\label{eq:subspace_ip}
    \mathbf{q}_r^\star = [\Phi_r^*\mathcal{F}(m)^*\Gamma^{-1}\mathcal{F}(m)\Phi_r + \gamma I_r]^{-1}\Phi_r^*\mathcal{F}(m)^* \Gamma^{-1}\mathbf{d}.
\end{equation}
The solution of this problem requires inverting a matrix only in $\mathbb{R}^{r_q\times r_q}$; the worst case computational complexities are thus $\mathcal{O}(dr_q^2 + r_q^3)$, and are independent of the grid representation of $\mathcal{Q}$. We approximate \eqref{eq:subspace_ip} via NEMO as
\begin{equation}\label{eq:subspace_ip_nemo}
    \mathbf{q}_r^\ddagger = [\mathbf{F}_\theta(\Psi_r^*m)^*\Gamma^{-1}\mathbf{F}_\theta(\Psi_r^*m) + \gamma I_r]^{-1}\mathbf{F}_\theta(\Psi_r^*m)^* \Gamma^{-1}\mathbf{d}.
\end{equation}
Here, $\mathbf{F}_\theta(\Psi_r^*m)^*$ denotes the Euclidean adjoint of the reduced operator, while $\Gamma^{-1}$ provides the observation-space weighting. By utilizing orthonormal bases, the restriction to the coefficient representations $\mathbf{q}_r = \Phi_r^*q$ inherits the inner product for $\mathcal{Q}$, preserving the $\mathcal{Q}$ inner product scaling. As noted previously, this formulation has other benefits in addition to computational savings. The subspace structure imposes a form of regularization by reconstructing the unknown parameter in a finite decomposition of smooth modes, as in truncated SVD approaches for inverse problems~\cite{vogel2002computational,hansen2010discrete}. The regularization parameter $\gamma$ can be determined using methods such as the L-curve (Pareto front) or Morozov discrepancy principle \cite{vogel2002computational}.

To assess the performance of NEMO inversion, we analyze the accuracy of the proposed NEMO-based inversion compared to using the true reduced p2o operator, as well as the complexity of neural operators needed to achieve a desired accuracy. Theorem~\ref{theorem:ip_ua} establishes results for NEMO inversion consistency and complexity requirements.

\begin{theorem}[NEMO Inversion Consistency and Dimension-Independence for Smooth Maps]\label{theorem:ip_ua}

\begin{subequations}\label{eq:noip_consistency}
Assume that $\mathbf{d} \in \mathbb{R}^{d}$, $\Gamma^{-1} \in \mathbb{R}^{d\times d}$. If $\mathcal{F} \text{ and } \mathbf{F}_\theta(\Psi_r^*(\cdot))\Phi_r^* \in L^\infty(\mathcal{M},\mu; \mathcal{L}(\mathcal{Q},\mathbb{R}^d))$, then
\begin{equation}\label{eq:linf_ipua}
\|\mathbf{q}_r^\star(m) - \mathbf{q}_r^\ddagger(m) \|_2 = \mathcal{O}\left(\|\mathbf{F}_\theta (\Psi_r^*m) - \mathcal{F}(m)\Phi_r\|_F \right) \text{ for a.e. } m.
\end{equation}
If $\mathcal{F} \text{ and } \mathbf{F}_\theta(\Psi_r^*(\cdot))\Phi_r^* \in L^2(\mathcal{M},\mu; \mathcal{L}(\mathcal{Q},\mathbb{R}^d))$, then
\begin{equation}\label{eq:l2_ipua}
    \mathbb{E}_{m\sim\mu}\left[\|\mathbf{q}_r^\star(m) - \mathbf{q}_r^\ddagger(m) \|_2\right] = \mathcal{O}\left( \sqrt{\mathbb{E}_{m\sim \mu}\left[ \|\mathbf{F}_\theta(\Psi_r^*m) - \mathcal{F}(m)\Phi_r\|_F^2\right]}\right).
\end{equation}
\end{subequations}
Additionally, assuming the map $\mathcal{F}$ has sufficient smoothness (as defined in Appendix~\ref{appendix:theorem-proof}), there exist ReLU neural networks with $\widehat{\mathcal{O}}(N)$ complexity (e.g., depth, size) yielding the following convergence of NEMO approximated inversion,
\begin{equation} \label{eq:no_dim_independence}
    \left( \|\mathbf{q}_r^\star(m) - \mathbf{q}_r^\ddagger(m) \|_2  \text{  or  } \mathbb{E}_{m\sim\mu}\left[\|\mathbf{q}_r^\star(m) - \mathbf{q}_r^\ddagger(m) \|_2\right] \right) =  \mathcal{O}\left(N^{-\alpha}\right),
\end{equation}
where $\alpha>0$ depends on the regularity of the map, and approximation setting, and $\widehat{\mathcal{O}}$ denotes limiting terms discarding logarithmic prefactors. In the setting of infinite-dimensional smooth maps described in Appendix~\ref{appendix:theorem-proof}, these rates are independent of the ambient discretization dimension.
\end{theorem}

See Appendix~\ref{appendix:theorem-proof} for the proof of this result.
Theorem~\ref{theorem:ip_ua} shows that the NEMO subspace inverse problem error, i.e. $\| \mathbf{q}_r^\star - \mathbf{q}_r^\ddagger\|_2$, is of the same order as the reduced p2o operator approximation error. This demonstrates the consistency of the NEMO learning problem and the inverse problem. Additionally, in the infinite-dimensional smooth map setting of Appendix~\ref{appendix:theorem-proof},
the required complexity of the NEMO neural network to achieve a particular subspace inverse problem accuracy is independent of the ambient discretization dimension, with the rate of convergence being determined by the smoothness of the map. We conclude that the NEMO approximation accuracy ensures consistent inverse problem solution accuracy, and the NEMO discretization-independence provides scalability for high-dimensional PDE problems. 

\paragraph{Uncertainty quantification}

In addition to the closed-form inverse solution, under interpretation of the regularization term as a Gaussian prior, the exact posterior covariance of the subspace inverse problem using the NEMO approximation has a closed-form expression, given by the inverse Hessian of the subspace least-squares objective,
\begin{equation}\label{eq:posterior-covariance}\mathbf{H}_r^{-1} = [\mathbf{F}_\theta(\Psi_r^*m)^*\Gamma^{-1}\mathbf{F}_\theta(\Psi_r^*m) + \gamma I_r]^{-1}.
\end{equation}
This covariance matrix is in $\mathbb{R}^{r_q\times r_q}$ and can be rapidly computed for real-time uncertainty quantification, and can be lifted to the original space via the basis expansion operator $\Phi_r$, yielding the covariance of the induced posterior supported on the reduced subspace. The availability of the posterior enables rapid uncertainty propagation, probabilistic analysis, and downstream informed decision-making within digital twin frameworks.

\section{Results}\label{sec:results}

NEMO outputs a reduced approximation of the true p2o operator $\mathcal{F}(m)$, given reduced bases $\Phi_r, \Psi_r$ for the inversion and model parameters, respectively. The approximation  $\mathbf{F}_\theta(\cdot)$, is used to obtain a closed-form solution of the inverse problem in the reduced subspace, $\mathbf{q}_r^\ddagger$. To assess the performance of NEMO, we test the accuracy of the NEMO output, and the associated inverse problem solution quality. 
To evaluate the subspace inverse problem solution, we compute the relative error for the subspace inverse problem solution between the true reduced p2o operator solution $\mathbf{q}_r^\star$ and the NEMO solution $\mathbf{q}_r^\ddagger$, given by
\begin{equation}
    \epsilon_{\mathbf{q_r}} = \frac{ \| \mathbf{q}_r^\ddagger - \mathbf{q}_r^\star\|_2}{\| \mathbf{q}_r^\star\|_2}
\end{equation}
We demonstrate the performance of NEMO in two PDE-based inverse problem applications: (1) contaminant transport initial condition identification, and (2) aerodynamic pressure load identification for hypersonics. All PDE solutions are solved with FEniCS~\cite{fenics}, while neural network training is performed using PyTorch. All NEMO models are fully-connected multi-layer perceptrons with ReLU or similar activation functions. The results are presented for a test set of inversion and model parameters which were not used in training NEMO.

\subsection{Contaminant transport initial condition identification} \label{section:contaminant}

We consider the goal of creating a digital twin for real-time emergency response to an airborne contaminant on the University of Texas at Austin campus. To build such a digital twin, we require the capability of rapid data assimilation to identify the conditions of the contaminant release. Mathematically, this task is posed as an inverse problem, where data from sparse sensor measurements are used to infer the initial condition of the contaminant. This would enable the digital twin to perform forecasting of the contaminant spread~\cite{akccelik2005dynamic, lieberman2013hessian} and inform subsequent action. To model the contaminant spread, we consider the advection-diffusion equation of the contaminant concentration $u(x,t)$ on domain $\Omega$ with boundary $\partial\Omega$ given by
\begin{equation}
    \frac{\partial u}{\partial t} + m\cdot \nabla u - \kappa \Delta u  = 0 \quad \text{in } \Omega \times (0,T)  \qquad  u(x,0) = q, \quad \nabla u(x,t) \cdot n = 0  \text{  on  } \partial \Omega \times (0,T).
\end{equation} 
where $m$ are the model parameters representing the wind velocity field, $q$ are the inversion parameters representing the unknown initial condition, and $\kappa$ is the diffusion coefficient. The domain is scaled such that 1 unit in the computational domain is equivalent to 100 meters. The tetrahedral mesh was produced with gmsh~\cite{geuzaine2009gmsh}, with 1,087,402 degrees of freedom.  The implementation for solving the PDE system is adapted from the advection-diffusion Bayesian inverse numerical example in \texttt{hIPPYlib} \cite{VillaPetraGhattas18}, which we refer to for further solver implementation details. We simulate up to time $T$ with 60 time steps. The observation locations were chosen manually, located on existing campus features. The contaminant observations consist of 50 sparse point measurements of the contaminant concentration, taken at three equally-spaced time instances $\frac{T}{3}, \frac{2T}{3}, T$, resulting in 150 total observations. Figure~\ref{fig:contaminant-problem} depicts the problem setup\footnote{Image credit (top left): Satellite view at 30°17'07.56"N 97°44'22.02"W, Google Earth, accessed 2025.}, including the domain, observations, and evolution in time.

\begin{figure}[h!]
    \centering
    \includegraphics[width=0.95\linewidth]{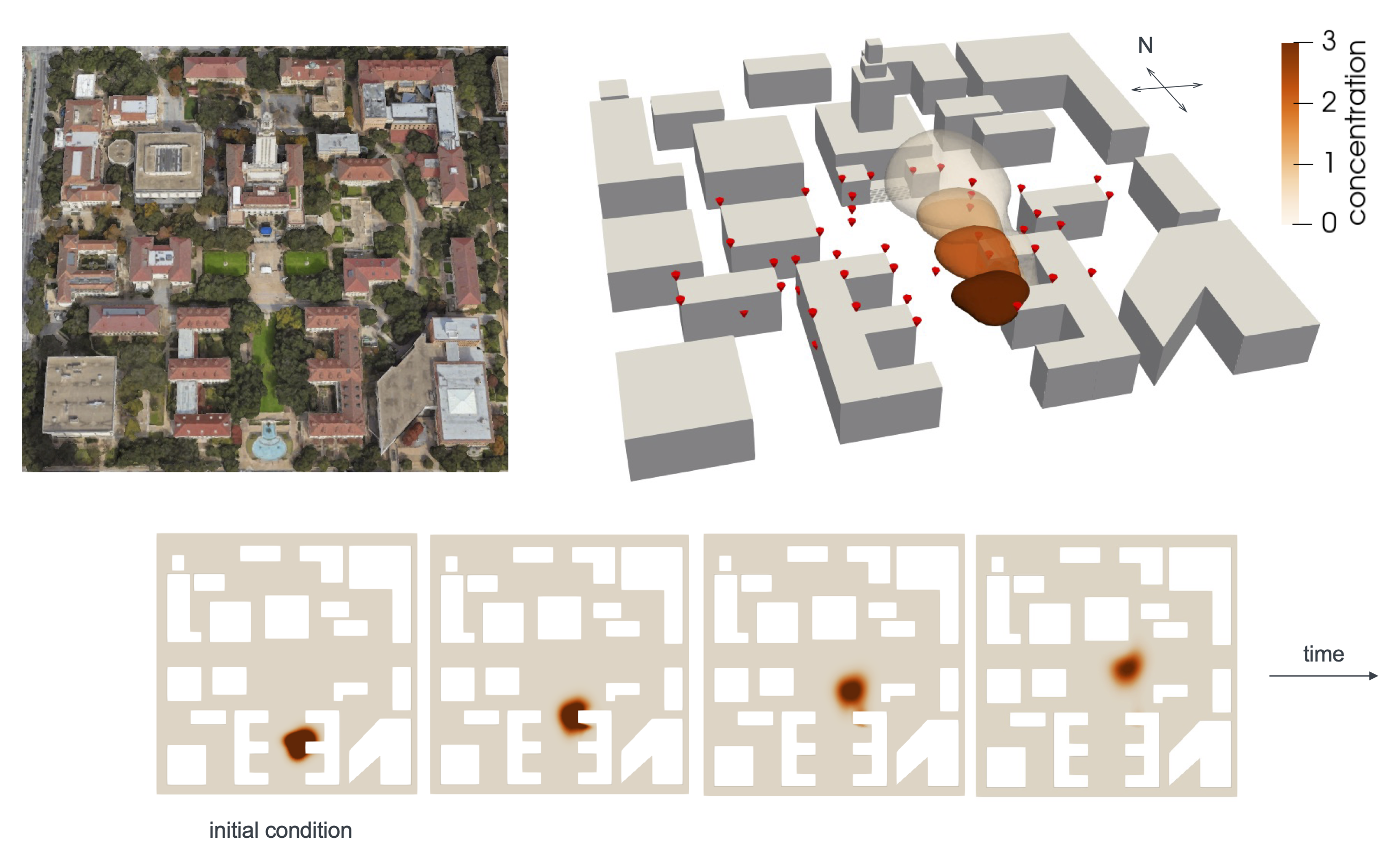} 
    \caption{\textbf{Contaminant initial condition transport.} We show the bird's-eye view of the Forty Acres at the University of Texas Austin campus (top left) and the corresponding computational representation (top right), with sensor locations depicted by the red markers. Iso-surfaces of an example non-dimensional contaminant initial condition and snapshots at times $t=\frac{T}{3}, \frac{2T}{3}, T$ for a given wind velocity field are illustrated in the computational domain, as well as the corresponding snapshots on a representative 2D slice of the domain (bottom).}
    \label{fig:contaminant-problem}
\end{figure}

In our problem setup, the wind velocity field is provided through external means (e.g. obtained via current weather models or measured via in situ sensors), but is not known ahead of time since it is dynamically changing.  For training NEMO, we compute representative velocity fields by solving a potential flow problem with parameterized boundary conditions. Sample ranges for these parameters were chosen to produce wind realizations that are representative of historical wind data \cite{winddata}, as well as realistic values of the P\'{e}clet number for atmospheric conditions. We solve the potential flow problem for 500 parameter samples and compute the reduced basis $\Psi_r$, choosing the model parameter subspace dimension $r_m = 3$. For the contaminant initial condition, we produce representative initial condition samples using clustered Gaussian distributions. The reduced basis $\Phi_r$ is computed from 1000 initial condition samples, and we choose the inversion parameter subspace dimension $r_q=50$. Figure~\ref{fig:contaminant-samples} depicts representative samples of the initial condition and wind velocity fields. We refer to Appendix~\ref{appendix:numerics-contaminant} for further numerical details on the initial condition and wind velocity sampling. Synthetic observation data are produced by solving the forward PDE for test initial conditions to obtain model observations, which are then corrupted with 1\% Gaussian noise.

\begin{figure}[h!]
    \centering
    \includegraphics[width=0.95\linewidth]{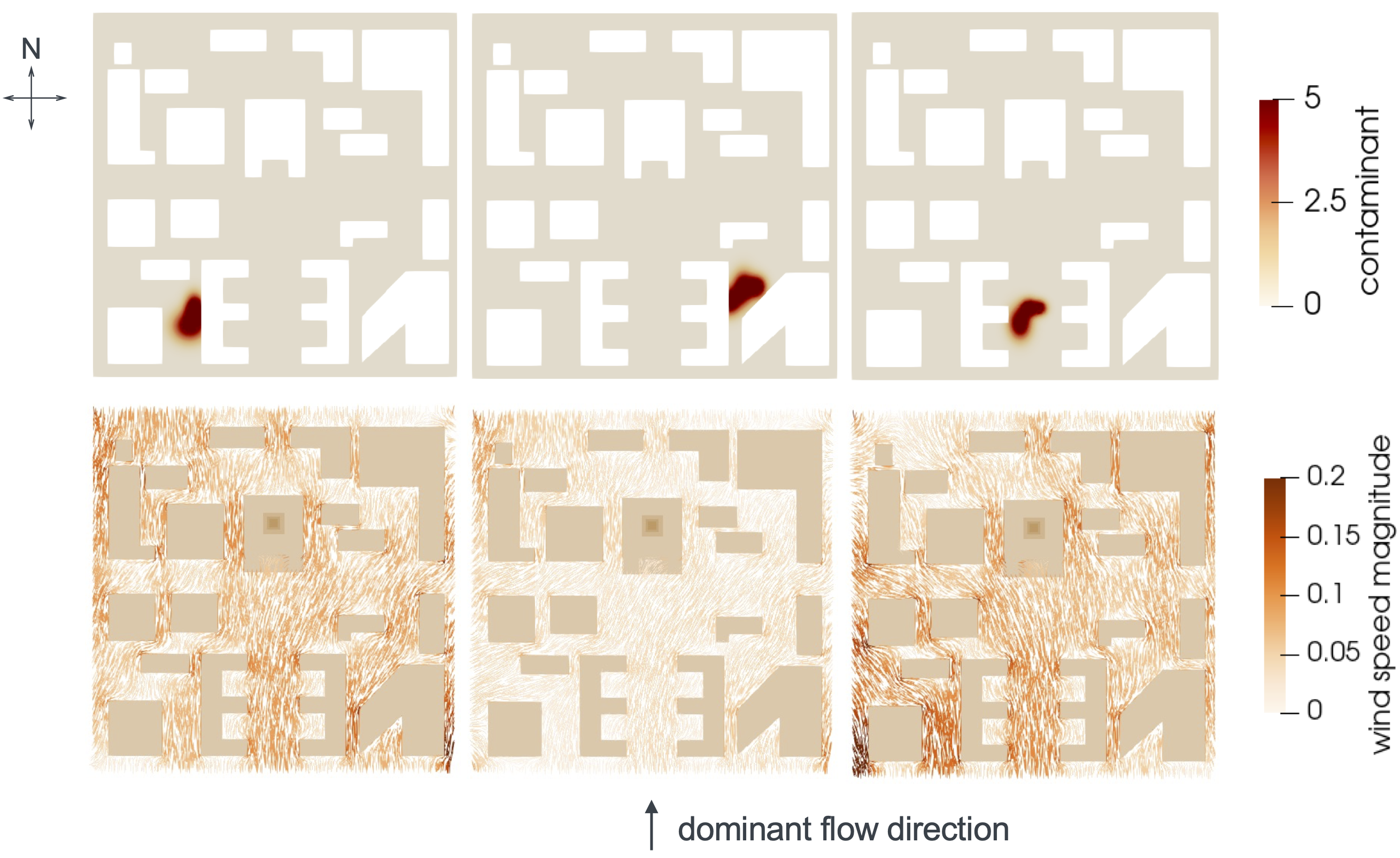}
    \caption{\textbf{Contaminant initial condition and wind velocity examples.} We illustrate example samples of representative 2D slices of the non-dimensional contaminant initial conditions (top) and wind velocity fields (bottom) which are used to compute a reduced basis for the inversion and model parameters, respectively.}
    \label{fig:contaminant-samples}
\end{figure}

We learn the reduced p2o operator $\mathbf{F}_\theta \in \mathbb{R}^{150 \times 50}$, resulting in an output dimension of $7500$. In this problem the NEMO has input dimension $r_m = 3$ with five hidden layers of dimensions (10, 500, 500, 1000, 2000) and utilizes GELU activation functions with fully connected layers. The models are trained for 10000 epochs using the PyTorch Adam optimizer with an initial learning rate of $5\times 10^{-4}$ and batch size 64, and using the PyTorch \texttt{ReduceLROnPlateau} adaptive learning rate with decay factor of $0.9$ and patience of 50 epochs. For training, we use a set of 400 model parameter samples for training/validation. We train the model using varying amounts of training data, and repeat the training over five independent experiments. Each experiment uses a different network parameter initialization and training data split. The test set consists of 100 independent samples of both the initial condition and velocity field that were not used in NEMO training. We chose the regularization parameter $\gamma$ based on the L-curve selection criteria for the given noise level on average on the test set. 

In Figure~\ref{fig:contaminant-nemo-performance}, we show the relative Frobenius error of the NEMO reduced p2o operator. The results are obtained from five separate training experiments evaluated on the test set. As the amount of training data increases, we observe smaller relative errors, demonstrating improved approximation of the true reduced p2o operator. For the inverse problem, we compare the NEMO inverse solution against the true reduced p2o operator solution, selecting the median-performing model from the five experiments. Figure~\ref{fig:contaminant-nemo-performance} also shows the relative errors in the subspace inverse problem solution for NEMO, $\mathbf{q}_r^\ddagger$, compared to the true reduced p2o operator solution, $\mathbf{q}_r^\star$. The NEMO inverse problem solution errors decrease as the approximation error decreases, demonstrating consistency between the inverse problem and the reduced p2o approximation performance. We visualize representative examples of ground truth initial conditions, and the corresponding inverse problem solutions using both the true reduced p2o operator and NEMO in Figure~\ref{fig:contaminant-nemo-performance}. We note that the colorbar range was truncated for visual clarity. Table~\ref{tab:stats} summarizes the problem dimensions and speedup of NEMO compared to evaluating the true reduced p2o operator using the PDE solver. NEMO obtains more than three orders of magnitude speedup, demonstrating its ability to accelerate inverse solutions for real-time applications.

\begin{figure}[h!]
    \centering
    \includegraphics[width=0.99\linewidth]{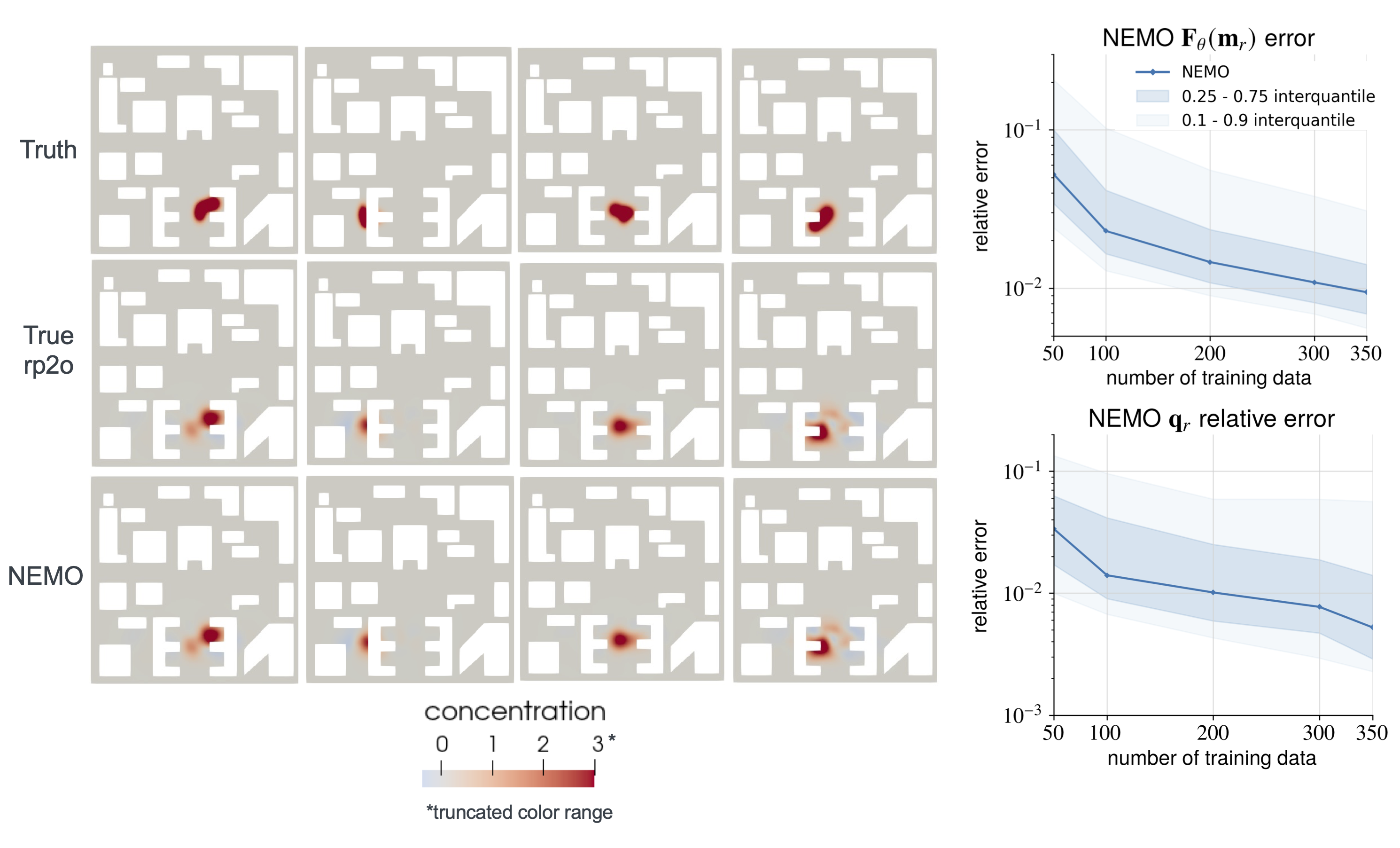}
    \caption{\textbf{Contaminant inversion with NEMO.} (Left) Examples of ground truth initial conditions, compared to the reconstructed initial conditions from the true reduced p2o and NEMO inverse problem solutions on a slice of the domain. The true reduced p2o and NEMO inverse problem solutions are nearly identical, and we observe good localization of the true contaminant initial condition. (Right) Relative errors for NEMO reduced p2o operator approximation, and subspace inverse problem solution. The NEMO subspace solution approaches the true reduced p2o operator solution as the NEMO approximation improves. }
    \label{fig:contaminant-nemo-performance}
\end{figure}

\begin{table} [h!]
\renewcommand{\arraystretch}{1.4}
\begin{center}
\caption{\textbf{Problem size and runtime statistics summary.} The true reduced p2o operators are computed in parallel on an AMD EPYC 7H12 with 64 cores using the PDE solver. NEMO is evaluated on a single core on the same machine, and obtains approximately 3-4 orders of magnitude speedup over the PDE solver. This highlights the speedup advantages of NEMO for real-time applications. }\setlength\tabcolsep{3mm}
\begin{tabular}{ cccccccc} \label{tab:stats}
\multirow{2}{*}{Problem}  & \multirow{2}{*}{$r_q$} & \multirow{2}{*}{$r_m$} & \multirow{2}{*}{$d$} & \multirow{2}{*}{\# NEMO parameters} 
& \multicolumn{2}{c}{\hspace{-3mm} Time to compute $\mathcal{F}(m)\Phi_r$} &  \multirow{2}{*}{Speedup} \\ 
 & & & & 
 & NEMO & True rp2o  &  \\ 
 \hline
 
 Contaminant & 50 & 3 & 150 & 1.78 x $10^7$ & 0.012 s  &  541 s &  $>$ 45,000  \\
 Hypersonic & 20 & 3 & 108 & 5.51 x $10^6$ & 0.006 s & 5.6 s & $ > $ 900 
\end{tabular} 
\end{center}
\end{table}

\subsection{Hypersonic aerodynamic pressure estimation}\label{section:aero}

We seek to create a digital twin of an aerospace vehicle that can provide adaptive real-time guidance and control. In this application, the aerodynamic pressure loads acting on the vehicle comprise key aerodynamic quantities of interest for guidance and control of the vehicle. To enable the desired digital twin, we require methods to rapidly characterize these pressure loads through real-time data assimilation during flight. For this example, we consider the hypersonic environment, where extreme temperatures on the external surface of the vehicle inhibit direct measurement of the pressure. To address this limitation, we seek to estimate the aerodynamic pressure field from measurements of the structural strain response -- a task mathematically posed as an inverse problem. A further challenge for hypersonics is that the extreme temperatures cause variations in the material properties of the structure.  Here, NEMO addresses the spatially varying, temperature-dependent material properties to enable real-time inversion for the aerodynamic loads.

In this application, the model parameters $m$ represent the temperature field, and the inversion parameters $q$ represent the surface pressure field. Assuming a quasi-static elastic response with small deformations, the displacement $u(x)$ is governed by 
\begin{equation}
    -\nabla \cdot \sigma(u) = f  \qquad u(x) = 0 \text{ on } \partial\Omega_\text{aft}
\end{equation}
where $\sigma$ is the stress tensor, $f=Cq$ maps the surface pressure inversion parameters $q$ to forces $f$, and $\partial\Omega_\text{aft}$ is the aft boundary of the vehicle. The constitutive law is given by 
\begin{equation}
    \sigma = \lambda(m) \text{tr}(\varepsilon)I + 2\mu(m)\varepsilon 
\end{equation}
where $\lambda(m), \mu(m)$ are the temperature-dependent material properties (Lam\'e parameters). The strain-displacement relation is given by $\varepsilon = \frac{1}{2}(\nabla u ^\top +  \nabla u)$. Here, we assume the (pressure) inversion parameter-induced strain is measured in isolation from thermal strain. The vehicle we consider is the IC3X conceptual hypersonic vehicle~\cite{witeof2014initial}, for which the strain-based estimation approach without temperature effects was first introduced in~\cite{pham2025}. The structural model is a modified geometry with no fins and hollow internal structure. The finite element model has a tetrahedral mesh produced using gmsh~\cite{geuzaine2009gmsh}, and contains 461,664 degrees of freedom, with fixed boundary conditions applied on the aft end of the vehicle. The observations $\mathbf{d}$ are pointwise measurements of the pressure-induced strain. The sensors are placed with even spacing between $x=[1.8, 3.4]$ with $0.2$ spacing, for a total of 9 $x$-locations. At each $x$-location, 8 sensors are evenly spaced around the inner circumference measuring in the streamwise direction, and an additional 4 sensors are evenly spaced measuring in the circumferential direction, for a total of 108 sensors. The subspace for the pressure loads is computed using Equation~\ref{eq:goal_oriented_F_reduction} with simulated pressure data over a range of flight conditions using CART3D~\cite{cart3d}, which solves the inviscid Euler equations. The flight conditions considered are Mach number in range $[5, 7]$, and angle of attack and sideslip angle in range $[-8, 8]$, at a constant altitude. We choose to retain $r_q = 20$ modes. For the temperature fields, we consider three modes which capture spatially-varying temperature fields as a result of heating related to the Mach number, angle of attack, and sideslip angle. As a result, the subspace dimension $r_m = 3$ by construction.

We learn the reduced p2o operator $\mathbf{F}_\theta \in \mathbb{R}^{108 \times 20}$, thus the output dimension is $2160$. For this problem, the NEMO has input dimension $r_m = 3$, with six hidden layers of dimension (10, 100, 300, 500, 1000, 2160) and utilizes softplus activation functions with fully connected layers. The models are trained for 2000 epochs using the PyTorch Adam optimizer with an initial learning rate of $5\times 10^{-4}$ and batch size 64, and using the \texttt{ReduceLROnPlateau} adaptive learning rate with decay factor of $0.95$. The training data are produced by sampling the pressure and temperature field reduced coordinates to obtain 1052 total samples, of which 242 samples were reserved for validation. In this problem, standard scaling was also used for the outputs. Training was conducted for varying amounts of training data and was repeated for five independent experiments, where each experiment utilized a different network parameter initialization.

\begin{figure}[h!]
    \centering
    \includegraphics[width=0.99\linewidth]{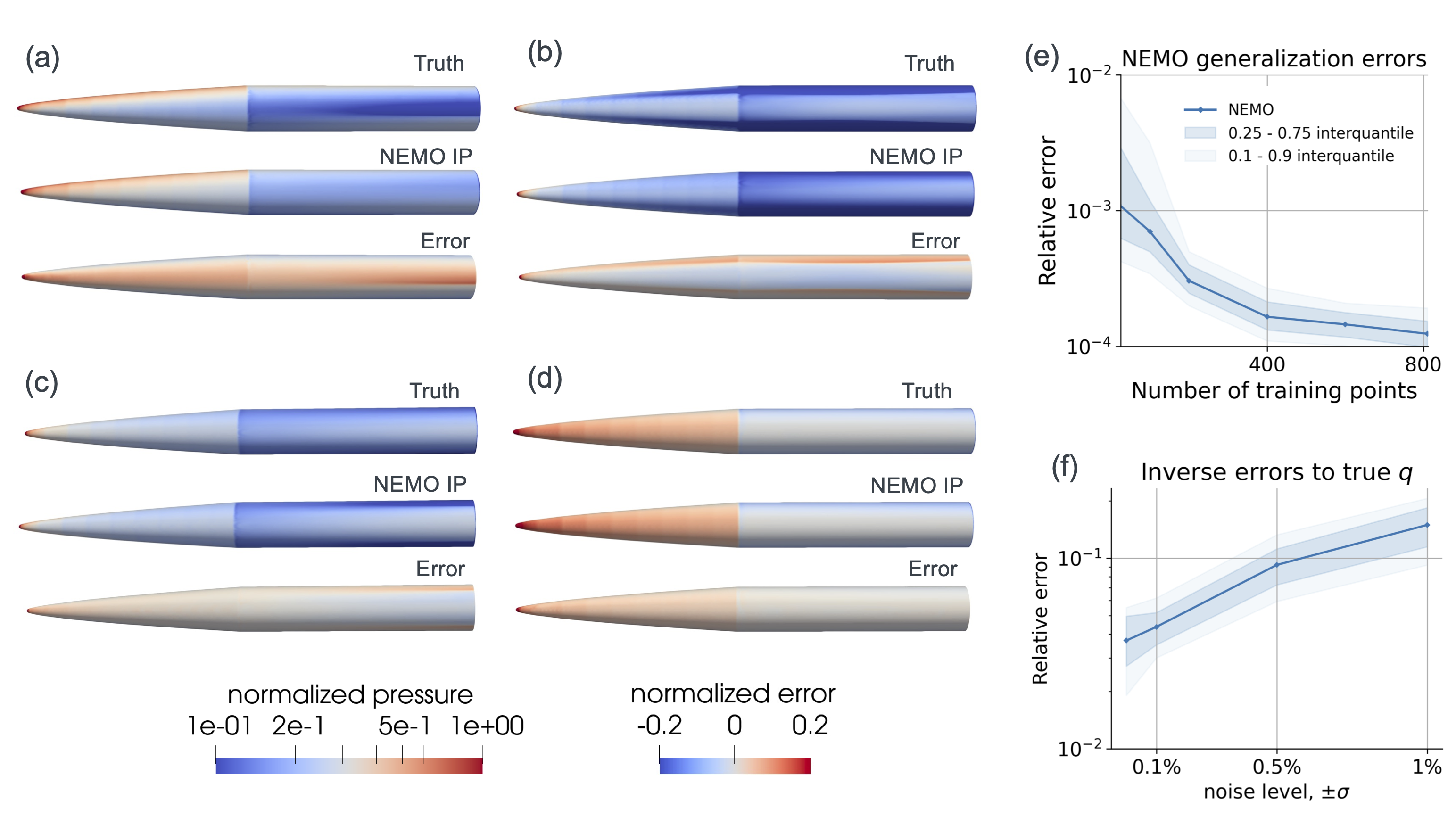}
    \caption{\textbf{Hypersonic aerodynamic pressure inversion with NEMO. (a-d)} Examples of ground truth pressure fields, compared to the corresponding reconstructed pressure from the NEMO inverse problem solution, and corresponding normalized error. The pressure is normalized such that the truncated visible upper range is one, and the relative error is computed pointwise. We observe the errors occur at regions of discontinuity on the true pressure field, due to the smoothing nature of the regularization in the inverse problem. \textbf{(e)} Relative errors in generalization for NEMO with respect to the true $\mathcal{F}(m)\Phi_r$. \textbf{(f)} Relative reconstructed inverse solution errors with respect to the truth, for various noise levels. }
    \label{fig:aero-ip-results}
\end{figure}

The test set consists of 50 pressure fields computed using the CFD model with flight conditions within the flight condition ranges stated above. Combinations of these 50 test pressures and 100 model parameter test samples are used to produce the test observation set, resulting in 5000 total test samples. We demonstrate the NEMO performance in Figure~\ref{fig:aero-ip-results}. First, we examine the NEMO relative Frobenius errors on the test set. Figure~\ref{fig:aero-ip-results}(e) shows the median relative errors, with interquantile ranges capturing the variation over the test set. These results are averaged over five independent training experiments. We observe decreasing relative errors with more training data, indicating improved approximation of the true reduced p2o operator. Second, for the inverse problem, we demonstrate the NEMO performance using the best-case model for three different noise levels: 0.1\%, 0.5\%, and 1\%, with regularization parameter $\gamma$ chosen separately for each noise level. Figure~\ref{fig:aero-ip-results}(f) shows the relative error of the NEMO inverse solution pressure field reconstruction compared to the true surface pressure field on the test set. We observe that the NEMO inverse solution achieves low relative errors, particularly in the low-noise regime. As the observational noise increases, we observe deteriorating performance benefits since the observation noise becomes the dominating error, requiring increased regularization in the inverse solution. Examples of the NEMO pressure field reconstructions for different test samples are visualized in Figure~\ref{fig:aero-ip-results}(a-d). Again, Table~\ref{tab:stats} reports the model size and speedup comparisons, demonstrating three orders of magnitude speedup compared to evaluating the true reduced p2o operator. For hypersonic vehicles, real-time is considered to be the typical onboard measurement frequency of 100 Hz. Here, NEMO evaluation meets this requirement on the given hardware, and can be deployed on optimized hardware to achieve the required real-time evaluation speeds.

\subsection{NEMO physics structure advantages compared to existing state-of-the-art}

\subsubsection{Forward surrogate comparison}

To assess the benefits of the closed-form solution of the NEMO method, we provide a comparison of the online computational cost compared to state-of-the-art neural operator methods for learning forward PDE surrogate models. As a representative case study, we select the multiple-input neural operator network (MIONet)~\cite{jin2022mionet}. The MIONet is a multiple-input variant of the state-of-the-art deep operator network (DeepONet)~\cite{lu2021} for forward PDE surrogates, which enables management of both the model and inversion parameter inputs in our problem formulation. The MIONet model, denoted by $\mathcal{G}_\theta$, consists of two branch networks for each input function in their reduced coordinate form, with equal-sized outputs that are combined via Hadamard (element-wise) product. The trunk network of MIONet is imposed as a linear layer which maps to the finite-dimensional observations, such that $\mathbf{d} = \mathcal{G}_\theta(\Phi_r^*q, \Psi_r^*m)$. To solve the inverse problem with $\mathcal{G}_\theta$, we solve a deterministic optimization problem with an equivalent subspace objective function using L-BFGS with a specified termination tolerance; further details on the comparison can be found in Appendix~\ref{appendix:comparison}. We note that the MIONet example is representative of other architectures for forward surrogates, which will also require solving the optimization online.

We assess both inverse problem performance and real-time computational cost to compare NEMO and MIONet. For both models, we observe comparable inverse problem reconstruction performance on the test set, since both models are sufficiently trained and utilize the same reduced basis. However, the optimization process using MIONet incurs longer online runtimes, and the computational complexity also varies over the test set measurements and regularization parameter $\gamma$ used. NEMO, in contrast, has a fixed computational cost that is independent of the measurements or parameters. Table~\ref{tab:comparison} summarizes the inverse problem timing results with both models on an NVIDIA Grace Hopper GH200 CPU; we note that on the GPU, the NEMO time improves even further due to parallelization of the matrix operations, while the optimization does not improve, since parallelization of the relatively small MIONet model is less significant than the increased overhead latency. 

We emphasize that the MIONet example is representative of other architectures for forward surrogates, which will also require solving the optimization online. Different architectures may change the cost of the forward evaluation, but does not impact the convergence behavior of the optimization (given that the forward map is sufficiently learned) as shown in Table~\ref{tab:comparison}, since the convergence behavior is dependent on the inverse parameter landscape.

\begin{table} [h!] 
\renewcommand{\arraystretch}{1.2}
\begin{center}
\caption{\textbf{State-of-the-art neural operator online cost comparison.} The NEMO inverse solution has a low, fixed online computational cost, while the MIONet model requires solving a nonlinear online optimization problem that does not have convergence guarantees, and costs vary with the regularization parameter $\gamma$, which changes the complexity of the inverse problem. The timing results were obtained on an NVIDIA Grace Hopper GH200 CPU. We observe speedups of more than an order of magnitude with NEMO over MIONet-based inversion. \vspace{3mm}}\setlength\tabcolsep{5mm}
\begin{tabular}{ ccc @{\hspace{0.75\tabcolsep}} c @{\hspace{0.75\tabcolsep}} cc} 
\multirow{2}{*}{Problem}  & \multirow{2}{*}{NEMO time (s)} & \multicolumn{3}{c}{Optimization with MIONet}  & \multirow{2}{*}{NEMO Speedup} \\ 
& & $\gamma$ &  avg \# iterations  & avg time (s) \\ 
 \hline
 \addlinespace[5pt]
 
 Contaminant & 0.0021 & 1e-1  &  15.1 &  0.022 & $ > $ 10  \\
                  & & 1e-2  &  28.8 &  0.037 & $ > $ 17  \\ 
                  & & 1e-3  &  51.6 &  0.077 &  $ > $ 36 \\ 
                  & & 1e-4  &  73.6 &  0.119 &  $ > $ 56\vspace{3mm} \\
             
 Hypersonic & 0.0017 & 1e-1  &  25.6 &  0.0336 & $ > $ 19  \\
                  & & 1e-2  &  35.1 &  0.0421 &  $ > $ 24  \\ 
                  & & 1e-3  &  43.3 &  0.0547 &  $ > $ 32 \\ 
                  & & 1e-4  &  44.5 &  0.0583 &  $ > $ 34\\
\label{tab:comparison}
\end{tabular} 
\end{center}
\end{table}

\subsubsection{Direct inverse map comparison}

Learning the direct inverse map is challenging due to the fact that learning the map from the observations to the inversion parameters is highly ill-posed. To address this challenge, regularization must be imposed on the direct inverse learning problem. This can be achieved through tuning the architecture of the learned inverse map, or by using training data that consists of regularized inverse solutions using PDE.
In the first case, there is a significant lack of interpretability since the regularization is imposed through the neural network architecture, rather than from a physics-informed perspective. In the second case, the problem structure is not being exploited, since a neural network is used to learn the least-squares solution for which we have a closed-form expression. 
Further, in both cases, the regularization must be built into the learned inverse map. This means that if the noise in the data changes, the model must be re-learned. NEMO, on the other hand, allows for principled tuning of the regularization parameter.

To illustrate, we train a direct inverse MIONet model which maps the observations and reduced model parameters to the true reduced inversion parameter solution, $\mathbf{q}_r^\star$. Figure~\ref{fig:comparison-direct-inverse} shows the distribution of the test relative error to the true reduced inverse problem solution $\mathbf{q}_r^\star$, with two levels of regularization $\gamma_1, \gamma_2$, where $\gamma_1 > \gamma_2$, for NEMO, the forward MIONet, and the direct inverse MIONet. The direct inverse MIONet is trained using the subspace inverse solutions which are regularized with $\gamma_1$. The MIONet comparisons are trained with the number of training data points equivalent to the number of forward solves required to produce the NEMO training data, i.e. $r_q$ times the number of samples of $m$ used for training NEMO. Further details on the direct inverse comparison can be found in Appendix~\ref{appendix:comparison}. First, observe that NEMO achieves the lowest errors compared to both MIONet models for the same amount of training data. However, all models achieve a majority of relative errors below 10\%. Second, we observe that for the inverse MIONet, the errors with $\gamma_2$ increases much more than the increase in the other models.  
This is because the regularization is fixed with the learned model, and a new model must be retrained if different noise levels emerge or different regularization is desired.  This shows the limited adaptability and lack of interpretability in the directly learned inverse map. We note that the evaluation of the direct inverse map is rapid, comparable to NEMO, with no online optimization required. The resulting inverse reconstruction error for the hypersonic example is also shown in Figure~\ref{fig:comparison-direct-inverse}. We observe the NEMO and forward MIONet reconstruction performance are comparable, although we emphasize the MIONet online cost is larger, as shown in Table~\ref{tab:comparison}. The direct inverse MIONet also performs similarly to NEMO at $\gamma_1$, however, the increased $\mathbf{q}_r^\star$ errors at $\gamma_2$ leads to increased reconstruction errors.

\begin{figure}[h!]
    \centering
    \includegraphics[width=0.99\linewidth]{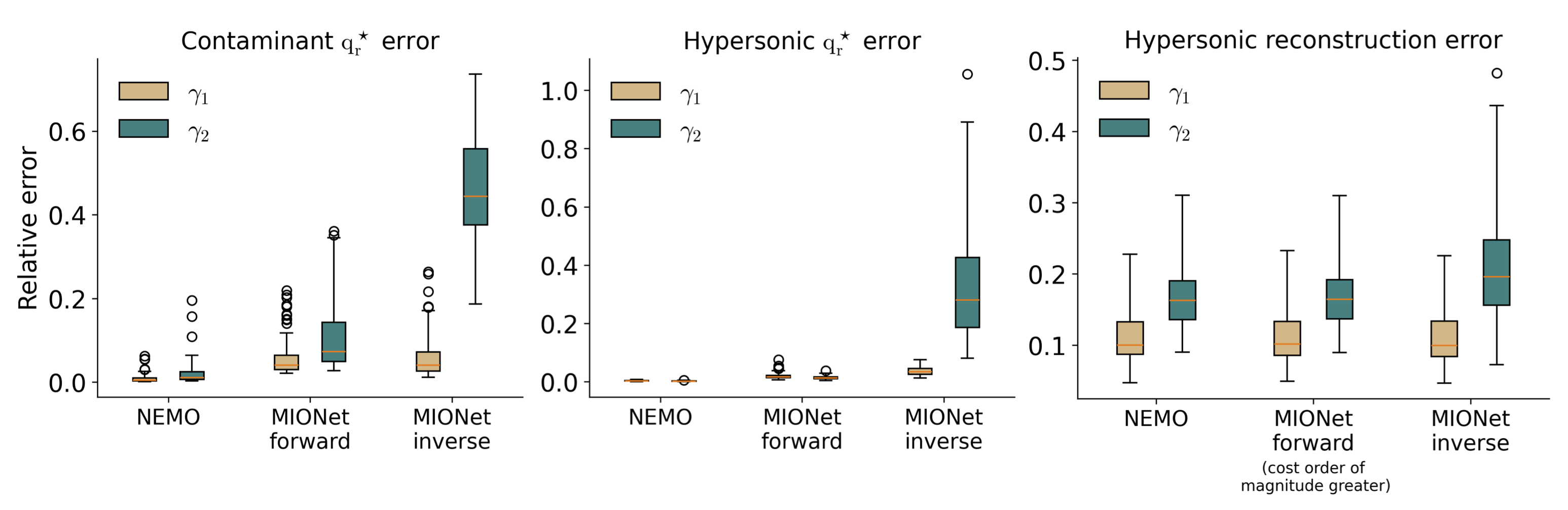}
    \caption{\textbf{Direct inverse map comparison.} (Left, middle) The direct inverse map is not robust to varying noise and regularization levels. We observe increased relative errors in the optimal reduced inversion parameters, $\mathbf{q}_r^\star$, with a change in the regularization parameter $\gamma_2$ compared to NEMO or the MIONet forward surrogate. (Right) The increased error in $\mathbf{q}_r^\star$ leads to increased reconstruction errors, as shown for the hypersonic example. }
    \label{fig:comparison-direct-inverse}
\end{figure}

\subsubsection{Uncertainty quantification}
Further, NEMO provides a direct means of uncertainty quantification for the inversion parameters. In the Bayesian setting, interpreting the regularization term as a Gaussian prior, the exact subspace posterior covariance for our target class of problems is given by the inverse Hessian of the subspace least-squares objective that characterizes the inverse problem. As a result of the NEMO formulation, this posterior covariance for the reduced inversion parameters has a closed-form expression that can be rapidly evaluated with the NEMO approximation. The eigenvectors of the posterior covariance provide the solution modes that are best informed by the data, visualized in Figure~\ref{fig:posterior-modes}. 
In contrast, generic nonlinear forward surrogate neural operator approaches typically require iterative sampling procedures (e.g. Markov Chain Monte Carlo), or linearization, to characterize the posterior. The subspace posterior provided by NEMO enables uncertainty propagation with live data streams, such as in Kalman filtering. This explicit characterization of the posterior enables risk-aware inference for noisy data, supports optimal experimental design by informing which sensors maximally reduce uncertainty, and provides probabilistic state estimates that enable broader digital twin frameworks~\cite{kapteyn2021probabilistic}.

\begin{figure}[h!] 
    \centering
    \includegraphics[width=0.99\linewidth]{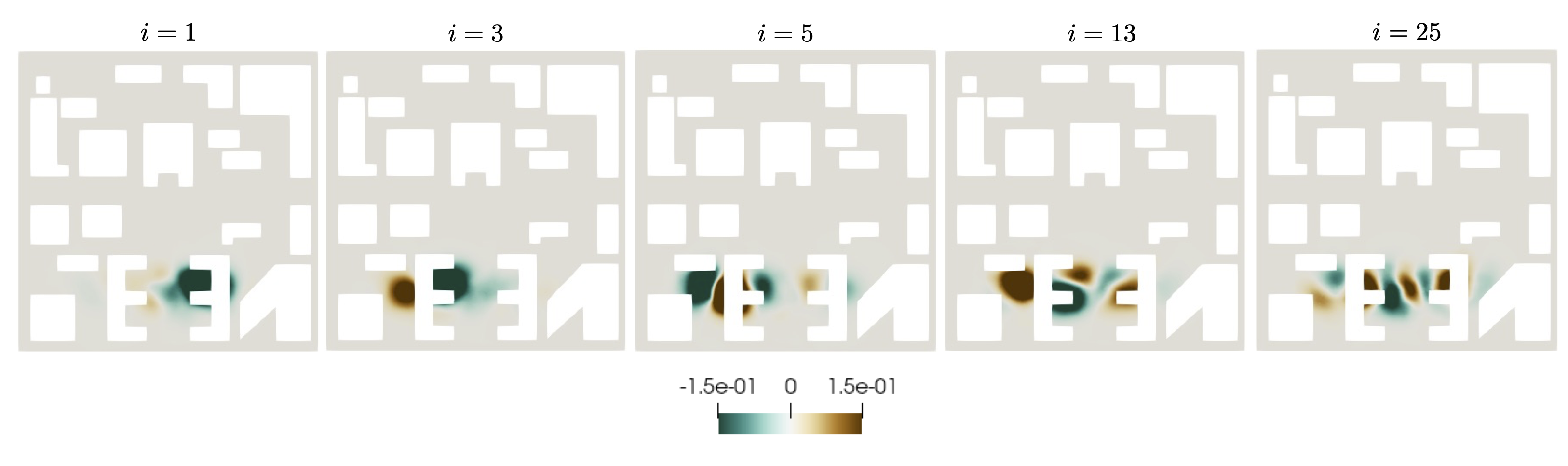}
    \caption{\textbf{NEMO uncertainty quantification.} The subspace posterior covariance can be computed explicitly with the NEMO approximation, providing direct, rapid uncertainty quantification for the inverse solution. This enables probabilistic analysis for informed decision-making in digital twins. Here, we visualize a selection of eigenmodes of the subspace Hessian at a particular $m$ for the contaminant initial condition problem, which have the least variance in the posterior. }
    \label{fig:posterior-modes}
\end{figure}

\subsubsection{Discussion}
 
We emphasize the advantages of NEMO compared to both (1) the traditional inverse problem solution using the full PDE system, and (2) other existing neural operator approaches for inverse problems. In the traditional PDE approach, the inverse solution is highly ill-posed and sensitive to regularization, and proper selection of the prior is critical for obtaining good quality solutions. This approach is not amenable to real-time deployment, since evaluation of the gradient and regularization tuning requires high-dimensional linear system solves that are computationally expensive. With a reduced p2o operator approach, the dimensionality reduction provides effective regularization in the inverse problem by restricting the solution to a low-rank subspace, resulting in improved solution quality. 

Further, other existing neural operator approaches which are surrogates for the forward PDE solution operator require solving a nonlinear optimization problem online to obtain the inverse solution. As demonstrated in Table~\ref{tab:comparison}, this optimization problem does not have specific structure, resulting in variable online costs without guarantees of convergence. Consequently, the availability of a closed form solution is a clear advantage of the NEMO approach for real-time solutions. For other neural operator approaches which learn the inverse map directly, the regularization is often imposed through the network architecture, and provides an arbitrarily nonlinear mapping which does not exploit the physics structure of the class of problems. NEMO, in contrast, exploits this physics structure, which also provides a means of uncertainty quantification, since the form of the inverse problem is equipped with an exact posterior covariance. As demonstrated, our NEMO approach has significant speedup, uncertainty quantification, and interpretability advantages compared to other neural operator approaches.

\section{Conclusions}\label{sec:conclusion}

The neural matrix operator (NEMO) methodology provides rapid inverse solutions to enable data assimilation for digital twins of physical systems governed by partial differential equations, with potential applications across a broad range of scientific domains. NEMO approximates the map from the model parameters $m$ to the parameter-to-observable operator $\mathcal{F}(m)$ in a reduced subspace, enabling a rapid-to-evaluate closed-form inverse solution. We demonstrate strong performance of NEMO in approximation for the numerical testbed examples in this work. We also demonstrate that NEMO provides orders of magnitude speedup compared to evaluating the true reduced p2o operator. We note that the training costs of NEMO are not a factor in assessing the real-time computational cost of the target application; they are offline costs that enable rapid evaluation speeds upon deployment. 

The efficiency and scalability of our framework is predicated on the viability of a low-rank approximation. The output dimension of NEMO scales with the subspace dimension of the inversion parameters and the number of observations. If either of these are large, the network size and complexity may increase. In particular, for time-dependent problems, the frequency of measurements may be high, resulting in large data dimension $d$ even for a small number of sensors. Using a single fully-connected network in this setting may result in very large networks. Increasing the dimension of the inversion parameter subspace may also increase the necessary size and complexity of NEMO. In these cases, more complex or advanced architectures may be considered. 

We further discuss the assumptions on the modeling parameters. In this work, we assumed the model parameters $m$ are fully characterized; in some cases, obtaining $m$ will require additional measurement or inspection, and possibly the solution of a separate inverse problem. Future work will conduct empirical analysis on the errors introduced in $m$ and their impact on NEMO performance. Future work may also explore nonlinear encodings of the model parameters. In the examples in this paper, the model parameters are well-represented in a low-dimensional subspace. For problems with large Kolmogorov n-widths (e.g., \cite{GreifUrban19}), NEMO can be extended to nonlinear encodings of the model parameters, while preserving the linear structure with respect to the inversion parameters.

\paragraph{Acknowledgements}

This work was supported by AFOSR grant FA9550-21-1-0089 under the NASA University Leadership Initiative (ULI), DOE grant DE-SC002317, and NSF grant 2313033. This material is based on work supported by the National Science Foundation Graduate Research Fellowship under Grant No.~DGE 2137420. The authors acknowledge the Texas Advanced Computing Center (TACC) at The University of Texas at Austin for providing computational resources that have contributed to the results in this paper. The authors would like to thank Dr.~Patrick Blonigan, Dr.~Noel Clemens, and the members of the FAST ULI team for the helpful discussions on the hypersonic application. The authors would additionally like to thank Dingcheng Luo for helpful discussions during the preparation of the manuscript.

\newpage
\bibliographystyle{ieeetr}
\bibliography{ref}

\newpage 

\appendix
\section{Overview of notation}

In this section we provide further detail on notations for completeness and clarity.

\paragraph{Overview of spaces $\mathcal{U},\mathcal{Q},\mathcal{M}$}

\begin{itemize}

    \item We consider PDE problems arising on spatial domains $\Omega \subset \mathbb{R}^2, \mathbb{R}^3,\dots$, taking the following form:
\begin{equation}
    A(m) u = Cq.
\end{equation}

    \item The state variable $u $ belongs to the separable Hilbert space $\mathcal{U}$, which we assume to be infinite dimensional. For example $\mathcal{U} = L^2(\Omega)$, $H^1(\Omega)$, defined as
    \begin{align*}
        L^2(\Omega) = \left\{f:\Omega \rightarrow \mathbb{R}: \|f\|_{L^2(\Omega)}^2 < \infty, \enskip \langle f,g\rangle_{L^2(\Omega)} = \int_\Omega f(x)g(x)dx\right\}\\
        H^1(\Omega) = \left\{f:\Omega \rightarrow \mathbb{R}: \|f\|^2_{H^1(\Omega)} < \infty, \langle f, g\rangle_{H^1(\Omega)} = \int_\Omega f(x)g(x) + \nabla f(x)^T\nabla g(x)dx \right\},
    \end{align*}
    where $\nabla$ denotes the spatial gradient.

    \item The inversion parameter $q$ belongs to the separable Hilbert space $\mathcal{Q}$, which we also consider to be infinite dimensional. For example $\mathcal{Q} = L^2(\partial\Omega)$ or $H^{\frac{1}{2}}(\partial \Omega)$ when the parameter $q$ represents a boundary condition. We note that the space of $\mathcal{Q}$ is often a consequence of the regularity of $\mathcal{U}$ and the associated PDE operator, e.g., via the trace theorem / Sobolev embedding theorem. Often times it may be the case that for a general variational PDE problem $\mathcal{Q}$ is not a Hilbert space, e.g., $\mathcal{Q} = H^{-\frac{1}{2}}(\partial\Omega)$. But in these cases there are Hilbert space subsets that are continuously embedded in the spaces, e.g., $L^2(\partial \Omega) \subset H^{-\frac{1}{2}}(\partial \Omega)$, so we choose to work with the Hilbert space subset. We are thereby able to utilize the inner product to construct orthonormal bases at the loss of missing some rougher, distributional representations of $q$. This is a suitable tradeoff for inverse problems as we tend to seek smooth $q$ anyways in order to make the inverse problem well-posed. For more information see Section \ref{appendix:coupling}. 

    \item The model parameter $m$ belongs to the separable Hilbert space $\mathcal{M}$, which we consider to be either finite dimensional or infinite-dimensional depending on the parametrization of the PDE. 

    \item We denote the topological dual of a Hilbert space using $'$, e.g., $\mathcal{U}'$.
\end{itemize}

\paragraph{Overview of (linear) operators}

\begin{itemize}
    \item  We denote the space of linear operators that map from Hilbert space $\mathcal{X}$ to Hilbert space $\mathcal{Y}$ by $\mathcal{L}(\mathcal{X},\mathcal{Y})$, and define the operator norm as follows:
    \begin{equation}
        \|A\|_{\mathcal{L}(\mathcal{X},\mathcal{Y})} = \sup_{0 \neq x \in \mathcal{X}} \frac{ \|Ax\|_{\mathcal{Y}}}{\|x\|_{\mathcal{X}}}.
    \end{equation}
    This is the infinite-dimensional analogue of the finite dimensional 2 norm $\|\cdot\|_2$ for matrices. Additionally we will utilize the Hilbert--Schmidt norm for operators, which can be defined through an orthonormal basis for $\{\psi_i\}_{i=1}^\infty$ with $\langle \psi_i,\psi_j\rangle_\mathcal{X} = \delta_{ij}$, as 
    \begin{equation}
        \|A\|_{HS(\mathcal{X},\mathcal{Y})} = \sqrt{\sum_{i=1}^\infty \|A\psi_i\|^2_{\mathcal{Y}}}.
    \end{equation}
    This is the infinite-dimensional analogue of the Frobenius norm $\|\cdot\|_F$ for matrices. 

    \item The PDE operator $A :\mathcal{M} \times \mathcal{U} \rightarrow \mathcal{U}'$ is nonlinear in $m$ and linear in $\mathcal{U}$. For each $m$, $A(m) \in \mathcal{L}(\mathcal{U},\mathcal{U}')$. Likewise the inverse operator $A(m)^{-1} \in  \mathcal{L}(\mathcal{U}',\mathcal{U})$, for each $m$.

    \item The observation operator $B\in \mathcal{L}(\mathcal{U},\mathbb{R}^{d})$.

    \item The coupling operator $C \in \mathcal{L}(\mathcal{Q},\mathcal{U}')$.
    \item The operator $\mathcal{F}(m) = BA^{-1}(m)C \in \mathcal{L}(\mathcal{Q},\mathbb{R}^d)$.
\end{itemize}

\paragraph{Bochner spaces with respect to probability measures}

\begin{itemize}
    \item We consider probability measures $\mu$ for functions in $\mathcal{M}$. We consider the Borel sigma algebra $\mathcal{B}(\mathcal{M})$, generated by open sets with respect to the norm $\|\cdot \|_{\mathcal{M}}$, which is induced by the inner product $\langle \cdot, \cdot \rangle_{\mathcal{M}}$. The measure space is thus $(\mathcal{M},\mathcal{B}(\mathcal{M}),\mu)$, where the measure $\mu: \mathcal{B}(\mathcal{M})\rightarrow [0,1]$. 

    \item For finite dimensional $\mathcal{M}$, we work directly with densities.

    \item For a potentially different measure space $(\mathcal{X},\mathcal{B}(\mathcal{X}))$, and a measurable function $T:\mathcal{M}\rightarrow \mathcal{X}$, we consider the pushforward measure $T_\sharp \mu$, where $T_\sharp \mu(B) = \mu(T^{-1}(B))$ for all $B \in \mathcal{B}(\mathcal{X})$. In other words if $m \sim \mu$, then $T(m) \sim T_\sharp \mu$. We utilize pushforwards briefly in Appendix \ref{appendix:theorem-proof}

    \item We consider $L^p$ Bochner spaces with respect to reference measures, e.g., $\mu$ over $\mathcal{M}$ for $1 \leq p < \infty$. If $\mathcal{G}:\mathcal{M}\rightarrow \mathcal{Y}$ is a measurable function, then we define the parametric Bochner space $L^p(\mathcal{M},\mu;\mathcal{Y})$ as follows:
    \begin{equation}
        L^p(\mathcal{M},\mu;\mathcal{Y}) = \left\{\mathcal{G}: \mathcal{M}\rightarrow \mathcal{Y} : \underbrace{\left(\mathbb{E}_{m\sim \mu}\left[\|\mathcal{G}(m)\|^p_\mathcal{Y}\right]\right)^\frac{1}{p}}_{L^p(\mathcal{M},\mu;\mathcal{Y}) \text{ norm}} < \infty \right\}.
    \end{equation}
    \item The expectation. e.g., for the map $\mathcal{G}$ is defined as follows:
    \begin{equation}
        \mathbb{E}_{m \sim \mu}\left[\mathcal{G}(m) \right] = \int_\mathcal{M} \mathcal{G}(m)d\mu(m)
    \end{equation}

    \item In the case of $L^\infty(\mathcal{M},\mu,\mathcal{Y})$, we define it as follows:
    \begin{equation}
        L^\infty(\mathcal{M},\mu;\mathcal{Y}) = \left\{\mathcal{G}: \mathcal{M}\rightarrow \mathcal{Y} : \underbrace{\text{ess sup}_{m \in \mathcal{M}}\|\mathcal{G}(m)\|_\mathcal{Y}}_{L^\infty(\mathcal{M},\mu;\mathcal{Y}) \text{ norm}} < \infty \right\}.
    \end{equation}
    where the essential supremum is an upper bound for the norm almost-everywhere (a.e.), i.e., everywhere except on sets of measure zero.

    \item We utilize these spaces to measure the operators $\mathcal{F}(m)$, $\mathcal{F}(\cdot)\Phi_r$, and $\mathbf{F}(\Psi_r^*m)$, and thereby identify $\mathcal{Y}$ with spaces of linear operators. 

    \item We utilize the shorthand notation $L^p_\mu$ when the associated Bochner space can be determined from context.

\end{itemize}

\section{On the generality of the coupling operator $C:\mathcal{Q}\rightarrow \mathcal{U}'$} \label{appendix:coupling}
In this section we discuss the coupling operator $C$, which allows us to represent right hand sides, boundary conditions and initial conditions. Abstractly, the coupling operator $C:\mathcal{Q}\rightarrow \mathcal{U}'$, since we consider the forward PDE operator in a variational sense, so $A(m):\mathcal{U} \rightarrow \mathcal{U}'$. We give examples below.

\subsection{Sources}

In the case of sources, we can directly assume that $\mathcal{Q} = \mathcal{U}'$, and the coupling operator is the identity map, $C = I_{\mathcal{U}'}$. Take for concreteness the Poisson equation driven by a source, with homogeneous Dirichlet boundary conditions:
\begin{equation}
    -\nabla \cdot (\exp(m)\nabla u) = q \quad \text{in } \Omega; \qquad  u = 0  \quad \text{on } \partial \Omega.
\end{equation}

\subsection{Boundary conditions} Here we give a short summary of some methods used to construct coupling (lifting) operators for boundary conditions to be interpreted as right hand sides as in \eqref{eq:fwd_map_system}. We begin with a construction for Dirichlet boundary conditions, which we motivate in an abstract variational sense. Consider the PDE with Dirichlet boundary conditions, and a source term given by 

\begin{equation}
    A(m) u = f_0 \quad \text{in }  \Omega; \qquad u = q \quad \text{on } \partial\Omega.
\end{equation}
Suppose that $A(m)$ is a second order elliptic operator, and that $\Omega \subset \mathbb{R}^n$ is bounded with Lipschitz boundary $\partial \Omega$. Our goal is to find $u \in H^1(\Omega)$ such that $u|_{\partial\Omega} = q$. If $q\in \mathcal{Q}=H^{\frac{1}{2}}(\partial \Omega)$, then by the trace theorem there exists a bounded right inverse of the trace
\begin{equation*}
    \mathcal{R}:H^{1/2}(\partial\Omega)\rightarrow H^1(\Omega),
        \qquad T\mathcal{R}q=q.
\end{equation*}
Writing 
\begin{equation*}
    u = w + \mathcal{R}q, \qquad w \in H^1_0(\Omega)
\end{equation*}
gives the homogeneous boundary problem
\begin{equation*}
    A(m)w = f_0 - A(m)\mathcal{R}q.
\end{equation*}
Thus Dirichlet boundary data may be represented as an effective right hand side via
\begin{equation*}
    C(m)q = -A(m)\mathcal{R}q
\end{equation*}
Note that this coupling depends on $m$ through $A(m)$.
For more information see \cite[Section 1.2]{demkowicz2023mathematical}.

Neumann and Robin boundary conditions can be handled directly as right hand sides in variational forms, e.g., denoting by $n$ the unit normal vector to the boundary $\partial\Omega$, the elliptic PDE system
\begin{equation}
    -\nabla \cdot (\exp(m)\nabla u) = f_0 \quad \text{in }\Omega; \quad \exp(m)\nabla u \cdot n = q \quad \text{ on } \partial \Omega
\end{equation}
can be re-written variationally as: find $u \in H^1(\Omega)$
\begin{equation}
    \int_\Omega \exp(m)\nabla u \cdot \nabla v dx = \int_{\Omega} f_0 v dx + \int_{\partial\Omega} q v ds \quad \forall v \in H^1(\Omega).
\end{equation}
The boundary term then defines a coupling operator $C:\mathcal{Q}\rightarrow \mathcal{U}'$ through
\begin{equation}
    \langle Cq,v\rangle_{\mathcal{U}',\mathcal{U}} = \int_{\partial\Omega}qvds,
\end{equation}
and the Neumann data enter directly as a right hand side in the variational problem. Pure Neumann conditions may require additional compatibility and uniqueness conditions.
Robin and mixed boundary conditions can be handled by similar constructions. We refer the reader to the following references for more information \cite{demkowicz2023mathematical,brennermathematical,necas2011direct}. We utilize Neumann boundary conditions in the hypersonic pressure estimation example in Section~\ref{section:aero}.

\subsection{Initial conditions}

We give a simple construction of the coupling operators for initial conditions. For simplicity, consider the abstract first order autonomous PDE system, and we assume $m$ is not time-dependent:
\begin{equation} \label{eq:abstract_dynamical_system}
    \frac{\partial u}{\partial t} + K(m)u = f,  \quad u(0) = u_0.
\end{equation}
In this setting, the state space $\mathcal{U}$ is often considered to be the temporal Bochner space for times $t \in (0,T)$, e.g., $\mathcal{U} = L^2((0,T);H^1(\Omega))$ together with sufficient temporal regularity for the initial condition $u(0) = u_0$ to be well-defined.
The norm is defined as follows:
\begin{equation}
    \|u\|_{L^2((0,T);\mathcal{Y})} = \sqrt{\int_0^T\|u(t)\|^2_\mathcal{Y} dt}
\end{equation}
The semigroup theory of linear dynamical systems establishes that the solution to the homogeneous version ($f=0$) of \eqref{eq:abstract_dynamical_system} takes the following general form
\begin{equation}
    u_\text{hom}(t) = \exp(-tK(m))u_0.
\end{equation}
See, for example, \cite{engel2000one,pazy2012semigroups} for more information. The full solution includes the nonhomogeneous (particular) solution due to the source $f$, which becomes $u = u_\text{hom}+u_p$:
\begin{align}
    u(t) &= \exp(-tK(m))u_0 + \int_0^t \exp(-(t-s)K(m))f(s)ds.
\end{align}
This is known as the Duhamel formula. We abstractly have a solution and coupling operator together. We can here interpret the solution operator at time t, to be $A_t^{-1}(m) = \exp(-tK(m))$, and then the coupling operator for the initial condition takes the form of identity, while the coupling for the right hand side involves a memory kernel integral.

One can similarly see the same effect when considering a time-integrated discrete system. Consider e.g., Crank--Nicolson for the homogeneous equation above, the update at iteration $k$ is
\begin{equation} \label{eq:CN_update}
    \mathbf{u}_{k+1}  = \left[I_\mathbf{U} + \frac{\Delta t}{2}\mathbf{K}(\mathbf{m}) \right]^{-1} \left[I_\mathbf{U} - \frac{\Delta t}{2}\mathbf{K}(\mathbf{m}) \right]\mathbf{u}_k = \mathbf{A}^{-1}(\mathbf{m})\mathbf{u}_k,
\end{equation}
where we use bold face to distinguish vector valued numerical quantities from the aforementioned limiting operators and functions. The solution at index $k$, as a function of the initial conditions is given then by
\begin{equation}
    \mathbf{u}_{k+1} = \mathbf{A}^{-(k+1)}(\mathbf{m})\mathbf{u}_0.
\end{equation}
The coupling operator for the discrete initial condition is the identity on the discrete state space. We utilize an initial condition mapping in the contaminant advection diffusion problem in Section~\ref{section:contaminant}.

\section{Analysis of neural operator-based inversion} \label{appendix:no}

\subsection{NEMO inversion consistency and dimension-independence for smooth maps}\label{appendix:theorem-proof}

In this section we present the proofs for the NEMO inverse problem consistency and discretization dimension independence as established in Theorem~\ref{theorem:ip_ua}. The first part of this theorem establishes the consistency of the NEMO-approximated inverse problem. The first sub-result \eqref{eq:linf_ipua} is formalized in Proposition \ref{prop:pointwise_error_bound}. The second sub-result \eqref{eq:l2_ipua} is formalized in Corollary~\ref{prop:l1l2_bound}, and follows directly from Proposition~\ref{prop:pointwise_error_bound}. The second part discusses approximation rates under various smoothness assumptions. In the infinite-dimensional smooth map setting, these rates are independent of the ambient discretization dimension.

\subsubsection{Analysis of Inverse Problem Errors}

Here, we analyze the errors of the inverse problem that are introduced by the neural operator approximation. We begin with a pointwise a priori error bound for the solution of the inverse problem via the surrogate error.

\begin{proposition}\label{prop:pointwise_error_bound}

For every $m$ for which the reduced operators are defined,
\begin{equation}
    \|\mathbf{q}_r^\star(m)-\mathbf{q}_r^\ddagger(m)\|_2
    \leq
    C
    \|\mathbf{F}_\theta(\Psi_r^*m)-\mathcal{F}(m)\Phi_r\|_F,
\end{equation}
where one may take
\begin{equation}
    C
    =
    \frac{3}{2}
    \frac{
        \sqrt{\|\Gamma^{-1}\|_2}
        \|\mathbf{d}\|_{\Gamma^{-1}} 
    }{\gamma}
    \leq
    \frac{3}{2}
    \frac{\|\Gamma^{-1}\|_2\|\mathbf{d}\|_2}{\gamma}.
\end{equation}

\end{proposition}

\begin{proof}
    We employ a pathwise perturbation argument related to classical perturbation analyses of regularized inverses \cite{gulliksson2000perturbation,golub1973differentiation}.
    For simplicity, define $\mathbf{F}_0(m) = \mathcal{F}(m)\Phi_r$, and 
    $\mathbf{F}_1(m) = \mathbf{F}_\theta(\Psi_r^*m)$. Let $\mathbf{H}(m)=\mathbf{F}_1(m)-\mathbf{F}_0(m)$, and for each $t \in [0,1]$, define $\mathbf{F}_t(m) = \mathbf{F}_0(m) + t\mathbf{H}(m)$. Consider the following quantity
\begin{equation*}
    \mathbf{q}_t(m)
    =
    \left(\mathbf{F}_t(m)^*\Gamma^{-1}\mathbf{F}_t(m)+\gamma I_r\right)^{-1}
    \mathbf{F}_t(m)^*\Gamma^{-1}\mathbf{d},
\end{equation*}
which interpolates between $\mathbf{q}_0(m)=\mathbf{q}_r^\star(m)$ and $\mathbf{q}_1(m)=\mathbf{q}_r^\ddagger(m)$. Differentiating the normal equations
\begin{equation*}
    \left(\mathbf{F}_t^*(m)\Gamma^{-1}\mathbf{F}_t(m)+\gamma I_r\right)\mathbf{q}_t(m)
    =
    \mathbf{F}_t^*(m)\Gamma^{-1}\mathbf{d}
\end{equation*}
with respect to $t$ gives
\begin{equation}\label{eq:perturbation_formula}
    \left(\mathbf{F}_t^*(m)\Gamma^{-1}\mathbf{F}_t(m)+\gamma I_r\right)\partial_t{\mathbf{q}}_t(m)
    =
    \mathbf{H}^*(m)\Gamma^{-1}(\mathbf{d}-\mathbf{F}_t(m)\mathbf{q}_t(m))
    -
    \mathbf{F}_t^*(m)\Gamma^{-1}\mathbf{H}(m)\mathbf{q}_t(m).
\end{equation}
Since $\mathbf{q}_t(m)$ minimizes
\begin{equation*}
    \frac12\|\mathbf{F}_t(m)\mathbf{q}-\mathbf{d}\|_{\Gamma^{-1}}^2
    +
    \frac{\gamma}{2}\|\mathbf{q}\|_2^2,
\end{equation*}
comparison with $\mathbf{q}=0$ yields

\begin{equation*}
    \frac{1}{2}\|\mathbf{F}_t(m)\mathbf{q}_t(m) - \mathbf{d}\|^2_{\Gamma^{-1}} + \frac{\gamma}{2}\|\mathbf{q}_t(m)\|^2_2 \leq \frac{1}{2}\|\mathbf{d}\|_{\Gamma^{-1}}^2,
\end{equation*}
from which we can derive $\|\mathbf{F}_t(m)\mathbf{q}_t(m)- \mathbf{d}\|_{\Gamma^{-1}}
    \leq
    \|\mathbf{d}\|_{\Gamma^{-1}}$ and $\|\mathbf{q}_t(m)\|_2 \leq \frac{1}{\sqrt{\gamma}}\|\mathbf{d}\|_{\Gamma^{-1}}$.
Moreover, we have the bound
\begin{equation}
    \left\|
    \left(\mathbf{F}_t(m)^*\Gamma^{-1}\mathbf{F}_t(m)+\gamma I_r\right)^{-1}
    \right\|_2
    \leq
    \frac1{\gamma}.
\end{equation}
Additionally, we have:
\begin{equation}
    \left\|
    \left(\mathbf{F}_t(m)^*\Gamma^{-1}\mathbf{F}_t(m)+\gamma I_r\right)^{-1}\mathbf{F}_t(m)^*\Gamma^{-\frac{1}{2}}
    \right\|_2
    = \max_i \frac{\sigma_i(\Gamma^{-\frac{1}{2}}\mathbf{F}_t(m)) }{\sigma_i(\Gamma^{-\frac{1}{2}}\mathbf{F}_t(m))^2 + \gamma},
\end{equation}
where $\sigma_i(A)$ denotes the $i^{th}$ singular value of the matrix A. Consider next, the function $f(\sigma) = \frac{\sigma}{\sigma^2 + \gamma}$, which has the derivative $f'(\sigma) = \frac{\gamma - \sigma^2}{(\sigma^2 + \gamma)^2}$, and attains a maximum for $\sigma = \sqrt{\gamma}$, which yields the bound:
\begin{equation}
    \left\|
    \left(\mathbf{F}_t(m)^*\Gamma^{-1}\mathbf{F}_t(m)+\gamma I_r\right)^{-1}\mathbf{F}_t(m)^*\Gamma^{-\frac{1}{2}}
    \right\|_2
    \leq \frac{1}{2\sqrt{\gamma}}.
\end{equation}
Therefore, combining these bounds with \eqref{eq:perturbation_formula} and the triangle inequality, we obtain
\begin{align}
    \|\partial_t{\mathbf{q}}_t(m)\|_2
    &\leq
    \frac{
        \sqrt{\|\Gamma^{-1}\|_2}\|\mathbf{H}(m)\|_2\|\mathbf{d}\|_{\Gamma^{-1}}
    }{\gamma}
    +
    \frac{
        \sqrt{\|\Gamma^{-1}\|_2}\|\mathbf{H}(m)\|_2\|\mathbf{d}\|_{\Gamma^{-1}}
    }{2\gamma} \\
    &=
    \frac{3}{2}
    \frac{
        \sqrt{\|\Gamma^{-1}\|_2}\|\mathbf{d}\|_{\Gamma^{-1}}
    }{\gamma}
    \|\mathbf{H}(m)\|_2.
\end{align}
Using $\|\mathbf{H}(m)\|_2\leq\|\mathbf{H}(m)\|_F$ gives the desired result via integration over $t\in(0,1)$
\begin{align}
    \|\mathbf{q}_r^\star(m)-\mathbf{q}_r^\ddagger(m)\|_2
    =
    \left\|\int_0^1\partial_t{\mathbf{q}}_t(m)dt \right\|_2
    \leq
    \int_0^1\|\partial_t{\mathbf{q}}_t(m)\|_2dt
    \leq
    \frac{3}{2}
    \frac{
        \sqrt{\|\Gamma^{-1}\|_2}
        \|\mathbf{d}\|_{\Gamma^{-1}}
    }{\gamma}
    \|\mathbf{F}_\theta(\mathbf{m}_r)-\mathcal{F}(m)\Phi_r\|_F.
\end{align}
    
\end{proof}

\noindent We proceed to on-average bounds over the reference measure over $\mathcal{M}$, $\mu$.

\begin{corollary}\label{prop:l1l2_bound}
If $\mathcal{F}  \in L^2(\mathcal{M},\mu; \mathcal{L}(\mathcal{Q},\mathbb{R}^d))$,  and $\mathbf{F}_\theta \in L^2(\mathbb{R}^{r_m},(\Psi_r^*)_\sharp\mu; \mathbb{R}^{d\times r_q})$ then

\begin{equation}
    \mathbb{E}_{m\sim\mu}\left[\|\mathbf{q}_r^\star(m) - \mathbf{q}_r^\ddagger(m) \|_2\right] = \mathcal{O}\left( \sqrt{\mathbb{E}_{m\sim \mu}\left[ \|\mathbf{F}_\theta(\Psi_r^*m) - \mathcal{F}(m)\Phi_r\|_F^2\right]}\right).
\end{equation}
\end{corollary} 

\begin{proof}
    We have the following:

    \begin{align*}
        \mathbb{E}_{m\sim\mu}\left[\|\mathbf{q}_r^\star(m) - \mathbf{q}_r^\ddagger(m) \|_2\right] &\leq C \mathbb{E}_{m\sim\mu}\left[\|\mathbf{F}_\theta(\Psi_r^*m)-\mathcal{F}(m)\Phi_r\|_F\right] \\
        &\leq C\sqrt{\mathbb{E}_{m\sim\mu}\left[\|\mathbf{F}_\theta(\Psi_r^*m)-\mathcal{F}(m)\Phi_r\|^2_F\right]},
    \end{align*}
    where the constant $C$ is defined in Proposition \ref{prop:pointwise_error_bound}.
\end{proof}

\subsubsection{Approximation of $\mathcal{F}$ by $\mathbf{F}_\theta$} \label{appendix:approximation_bounds}

In Theorem \ref{theorem:ip_ua}, we bound the inverse problem error by different errors in the neural operator approximation for $L^2_\mu$ and $L^\infty_\mu$ Bochner spaces. In this section, we discuss different results that can be used to derive approximation rates for the neural operator learning in these particular spaces.
There are a large variety of general universal approximation results that can be used to prove the existence of neural network sequences that can achieve arbitrary errors in various norms. For finite-dimensional $\mathcal{M}$, there are a number of results that prove universal approximation of neural networks in $L^\infty$ restricted to compact sets $K \subset \mathcal{M}$ \cite{cybenko1989approximation,leshno1993multilayer,lu2017expressive}. When the reference measure has bounded moments, these results can be used to extend these results to $L^2_\mu$ via cutoff-type arguments \cite{barron1993universal,hornik1991approximation}.
In this subsection, we discuss results from the literature that can be used to not only prove the existence of the neural networks satisfying these universal approximation conditions, but also derive approximation rates and associated complexities for neural network approximation of $\mathcal{F}(m)\Phi_r$. 

For simplicity, we consider the vectorized output of $\mathbf{F}_\theta(\mathbf{m}_r)\in \mathbb{R}^{ d \times r_q}$ , $\text{vec}(\mathbf{F}_\theta)(\mathbf{m}_r) \in \mathbb{R}^{dr_q}$. Additionally we consider the case that $m$ is either finite-dimensional, or has a truncated representation on an orthonormal basis expansion as in \eqref{eq:qr_mr_expansion}.

\paragraph{Finite-dimensional (compact) $\mathcal{M} \subset \subset \mathbb{R}^{d_M}$.}

We first consider $\mathcal{M}$ to be a compact subset of a $\mathbb{R}^{d_M}$, and utilize a result from \cite{petersen2024mathematical}. For this set of results, we do not consider the dimension reduction operator $\Psi_r^*$, since $m$ is finite dimensional. We note that when $d_M$ is very large, dimension reduction may be necessary, and the errors due to the truncation can be addressed via e.g., Lipschitz constants pre-multiplying the basis truncation errors as in \cite{bhattacharya2021}.We proceed with a definition. 

\begin{definition}
    Let $k \in \mathbb{N}_0$ and $s \in [0,1]$, and $\mathcal{M} \subset \mathbb{R}^{d_M}$. Then for $f:\mathcal{M} \rightarrow \mathbb{R}$

    \begin{align}
        \|f\|_{C^{k,s}(\mathcal{M})} =& \sup_{m \in \mathcal{M}} \max_{\mathbf{\alpha} \in \mathbb{N}_0^{d_M}, |\mathbf{\alpha}|\leq k}|D^\mathbf{\alpha}f(m)| \nonumber \\
        &+ \sup_{m\neq y \in \mathcal{M}} \max_{\mathbf{\alpha}\in \mathbb{N}_0^{d_M}, |\mathbf{\alpha}|\leq k}\frac{|D^\mathbf{\alpha}f(m) - D^{\mathbf{\alpha}}f(y)|}{\|m - y\|^s_\mathcal{M}}
    \end{align}
and denote by $C^{k,s}(\mathcal{M})$, the set of functions $f \in C^k(\mathcal{M})$ for which $\|f\|_{C^{k,s}(\mathcal{M})} < \infty$.
\end{definition}
\noindent We now state the result that we extend to our case.
\begin{proposition}{Theorem 7.10 in \cite{petersen2024mathematical}} \label{prop:thm710pz}
    Let $d_M\in \mathbb{N}$, $k \in \mathbb{N}_0$, $s \in [0,1]$, and $\mathcal{M}=[0,1]^{d_M}$. Then there exists a constant $C >0$ such that for every $f \in C^{k,s}(\mathcal{M})$, and every $N\geq 2$ there exists a ReLU neural network $\mathcal{N}_N^f$ such that 
    \begin{equation}
        \sup_{m \in \mathcal{M}} | f(m) - \mathcal{N}_N^f(m)| \leq C \|f\|_{C^{k,s}(\mathcal{M})}N^{-\frac{k+s}{d_M}},
    \end{equation}
and $\text{size}(\mathcal{N}_N^f)\leq CN\log(N)$ and $\text{depth}(\mathcal{N}_N^f)\leq C\log(N)$.
\end{proposition}

\begin{corollary}
Consider the setting of Proposition \ref{prop:thm710pz}. Assume that $\mathcal{F}(\cdot)\Phi_r \in C^{k,s}(\mathcal{M})^{d\times r_q}$, e.g., a matrix valued function where each entry is a $C^{k,s}(\mathcal{M})$ function, then there exists a stacked ReLU neural network $\mathbf{F}_\theta$ with $dr_q$ outputs such that
\begin{equation}
    \sup_{m \in \mathcal{M}} \|\mathcal{F}(m)\Phi_r -  \mathbf{F}_\theta(m)\|_{\ell^\infty(\mathbb{R}^{d\times r_q})} \leq CN^{-\frac{k+s}{d_M}} \sup_{i,j} \|(\mathcal{F}(\cdot)\Phi_r)_{ij}\|_{C^{k,s}(\mathcal{M})},
\end{equation}
where $\text{size}(\mathbf{F}_\theta) \leq C d r_q N\log(N)$  and $\text{depth}(\mathbf{F}_\theta) \leq C\log(N)$.
\end{corollary}

\begin{proof}
    As a consequence of Proposition \ref{prop:thm710pz} we get the existence of $dr_q$ different neural networks for each entry in the coefficient matrix $\mathcal{F}(m)\Phi_r$ satisfying the bounds and complexities stated in Proposition \ref{prop:thm710pz}. What is left to establish is that they can be stacked together into one neural network, due to their potentially varying depths. 

    Find the set of neural networks with maximal depths. These networks have $\text{depth}_\text{max} = \mathcal{O}(\log(N))$ depth and $\mathcal{O}(N\log(N))$ complexity as a result of Proposition \ref{prop:thm710pz}. For each network with shorter depth, the depth difference is $L = \text{depth}_\text{max} - \text{depth}_i$. Since we consider ReLU neural networks, we can use the fact that ReLU neural networks can exactly reconstruct the identity (see for example Lemma 5.1 in \cite{petersen2024mathematical}). Such a network has depth $L$ and complexity $2L\times r_\text{hidden}$, where $r_\text{hidden}$ is the number of hidden neurons in the final layer of the shorter network (in this case it is 1, since the network is scalar-valued). Concatenating (composing) ReLU neural networks remains a ReLU neural network. Moreover, stacking up scalar valued neural networks to construct a matrix-valued neural network is still a ReLU neural network. Each of these neural networks has now the same depth, and the complexities are of the same order. Since there are $dr_q$ of them the size of the final stacked and depth-extended matrix-valued ReLU network is $\mathcal{O}(dr_q)\times \mathcal{O}(N\log(N)) = \mathcal{O}(dr_qN\log(N))$.

\end{proof}

\paragraph{Infinite-dimensional $\mathcal{M}$.}

For the infinite-dimensional case, we apply some results from \cite{herrmann2024neural}, which apply more broadly than the case we consider. In the present case we are interested in orthonormal systems for reducing the inputs (i.e., the operator $\Psi_r^*$) which we are mapping to finite-dimensional outputs (i.e., coefficient matrices in $\mathbb{R}^{d\times r_q}$). The results in \cite{herrmann2024neural} apply more generally to Banach spaces (via Riesz bases / frames) as well as infinite-dimensional outputs. For a more detailed exposition, we refer the reader to \cite{herrmann2024neural}, we give a brief overview for completeness.

In order to derive discretization dimension independent approximation rates, we restrict our attention to the map on scaled subsets of $\mathcal{M}^s\subset \mathcal{M}$. This identifies smooth subsets of the space $\mathcal{M}$. We define the weights $\mathbf{w} = (w_j)_{j\in\mathbb{N}}$ such that $\mathbf{w}^{1+\epsilon} \in \ell^1(\mathbb{N})$ for all $\epsilon>0.$ We can then define the $\mathcal{M}^s$ norm as follows.

\begin{equation}
    \|m\|_{\mathcal{M}^s}^2 = \sum_{j=1}^\infty \langle m,\psi_j\rangle_\mathcal{M}^2 w_j^{-2s}.
\end{equation}
And we define $\mathcal{M}^s = \{ m \in \mathcal{M}: \|m\|_{\mathcal{M}^s} < \infty \}$. We define the following operation on uniform intervals $U = [-1,1]^\mathbb{N}$, with scalar parameters $\rho>0$ and $s > \frac{1}{2}$
\begin{equation}
    \sigma_\rho^s = \begin{cases}
                    U \rightarrow \mathcal{M} \\
                    \xi \mapsto \rho \sum_{j \in \mathbb{N}} w_j^s \xi_j \psi_j. 
                \end{cases}
\end{equation}
Consider the product uniform measure over $U$, $\nu = \bigotimes_{j=1}^\infty
\text{Unif}_{[-1,1]}$. Then $\widetilde{C}^s_\rho(\mathcal{M}) = \{ \sigma_\rho^s(\xi): \xi \in U\}$ is a weighted infinite-dimensional cube with decaying side lengths. The approximation results rely on a technical assumption that uses complex extensions of the input and output spaces, e.g., $\mathcal{M}_\mathbb{C}^s = \{1,i\}\otimes \mathcal{M}^s$. 

\begin{assumption}[Assumption 3.3 in \cite{herrmann2024neural}, adapted]
\label{assumption:jakob33}
There exist $s>1$, $\rho>0$, $M<\infty$, and an open set
$O_{\mathbb{C}}\subset\mathcal{M}_{\mathbb{C}}$ containing
$\widetilde{C}_\rho^s(\mathcal{M})$ such that
\begin{equation}
    \sup_{m\in O_{\mathbb{C}}}
    \|\mathcal{F}(m)\Phi_{r_q}\|_F
    \leq M,
\end{equation}
and $\mathcal{F}(\cdot)\Phi_{r_q}:
    O_{\mathbb{C}}
    \rightarrow
    \mathbb{C}^{d\times r_q}$
is holomorphic.
\end{assumption}

As discussed in \cite{herrmann2024neural}, this assumption is satisfied by some solution operators for second order elliptic PDEs, which are among the solution operators that we are interested in for the present work. In our case we also conjugate these solution operators with linear operators $B,\Phi_r$ which map them to finite dimensional coefficient outputs. 
This setup allows us to derive both $L^\infty(\mathcal{M}^s, (\sigma_\rho^s)_\sharp\nu;\mathbb{R}^{d \times r_q})$ and $L^2(\mathcal{M}^s, (\sigma_\rho^s)_\sharp\nu;\mathbb{R}^{d \times r_q})$ results. 

Note for asymptotic convergence, more coefficients of the input function $m$ are required, so these results implicitly require $r(N)$ to grow. 

\begin{proposition}{Corollary 3.8 in \cite{herrmann2024neural}, adapted}
Suppose Assumption \ref{assumption:jakob33} is satisfied with $s>1$. Fix $\delta>0$ and $\rho>0$, then there exists a constant $C>0$ such that for every $N \in \mathbb{N}$ there exists a ReLU neural network $\mathbf{F}_{\theta,N}$ of size $\mathcal{O}(N)$,
depending on finitely
many input coefficients, such that
\begin{equation}
    \sup_{m\in B_\rho(\mathcal{M}^s)}
    \left\|
        \mathcal{F}(m)\Phi_{r_q}
        -
        \mathbf{F}_{\theta,N}
        \bigl(\Psi_{r(N)}^*(m)\bigr)
    \right\|_F
    \leq
    C N^{-(s-1)+\delta}.
\end{equation}

\end{proposition}
The number $r(N)$ of active input coefficients is finite for each $N$ but is
not required to remain fixed as $N\to\infty$.
We next state an improved mean-square approximation rate for the specific
product-induced measure
$\widetilde{\mu}=(\sigma_\rho^s)_\sharp\nu$.

\begin{proposition}{Theorem 3.10 in \cite{herrmann2024neural}, modified.}
    Assume that $\{\psi_j\}_{j \in \mathbb{N}}$ is an ONB of $\mathcal{M}$ and let Assumption \ref{assumption:jakob33} be satisfied with $s>1$. Fix $\delta>0$ and $\rho>0$. Then there exists a constant $C>0$ such that for every $N\in \mathbb{N}$ there exists a ReLU neural network $\mathbf{F}_{\theta,N}$ of size $\mathcal{O}(N)$ such that
    \begin{equation}
        \mathbb{E}_{m\sim\widetilde{\mu}} \left[ \|\mathcal{F}(m)\Phi_{r_q} - \mathbf{F}_{\theta,N}(\Psi_{r(N)}^*m)\|_F^2 \right] \leq C N^{-2(s-\frac{1}{2} )+2\delta},
    \end{equation}
where $\widetilde{\mu} = (\sigma_\rho^s)_\sharp\nu$.
\end{proposition}

\subsection{Analysis of dimension reduction errors} \label{appendix:dim_reduction}

In this section we discuss some errors due to the dimension reduction of $\mathcal{F}$. For notational simplicity we use $r$ subscript for both $\Phi_r$ and $\Psi_r$, noting that these rank values $r_q$ and $r_m$ are potentially different.

\paragraph{Dimension reduction for $\mathcal{F}$}

In this section, we focus on dimension reduction strategies for $\mathcal{F}$. We begin with the goal-oriented reduction of $\mathcal{Q}$, which is accomplished using the orthonormal basis computed from the following generalized eigenvalue problem

\begin{equation}\label{eq:gevp_f}
    \mathbb{E}_{m \sim \mu} \left[\mathcal{F}(m)^*\mathcal{F}(m) \right]\phi_i = \lambda_i \phi_i.
\end{equation}
For simplicity we omit the $\Gamma^{-1}$ weighting, however this can be easily accounted for by subsuming a factor of $\Gamma^{-\frac{1}{2}}$ into $\mathcal{F}$ on the left as in \eqref{eq:f_fans_thm}. This generalized eigenvalue problem arises via the following minimization problem:

\begin{equation} \label{eq:basis_min_f}
    \min_{\text{rank }r_q\text{ ONB }U_r}\mathbb{E}_{m \sim \mu} \left[\| \mathcal{F}(m)(I_\mathcal{Q} - U_rU_r^*)\|_{HS(\mathcal{Q},\mathbb{R}^d)}^2 \right].
\end{equation}
Note that since the range of $\mathcal{F}(m)$ is finite dimensional it is automatically a Hilbert--Schmidt operator. 

The fact that the basis from \eqref{eq:gevp_f} is the solution of \eqref{eq:basis_min_f} is a consequence of a generalization \cite{bhattacharya2021} of Fan's Theorem \cite{fan1949theorem}, which states , for a given $r$,

\begin{equation}
    \max_{\text{rank r ONB }U_r} \sum_{i=1}^r\left\langle \mathbb{E}_{m \sim \mu} \left[\mathcal{F}(m)^*\mathcal{F}(m)\right]u_i,u_i\right\rangle_\mathcal{Q} = \sum_{i=1}^{r} \lambda_i.
\end{equation}
This result can be used to prove that $\Phi_r$ minimizes the error in \eqref{eq:basis_min_f}. This yields the following output reduction error:
\begin{equation}
    \mathbb{E}_{m \sim \mu} \left[\| \mathcal{F}(m)(I_\mathcal{Q} - \Phi_r\Phi_r^*)\|_{HS(\mathcal{Q},\mathbb{R}^d)}^2 \right] = \sum_{j>r_q} \lambda_j.
\end{equation}
Thus $\Phi_r$ provides an optimal linear reduction of $\mathcal{Q}$ for the reconstruction of $\mathcal{F}$, in expectation with respect to $\mu$. Using this basis for the training of the neural operator we can bound errors in the full space through the following error decompositions in $L^2(\mathcal{M},\mu;\mathcal{L}(\mathcal{Q},\mathbb{R}^d))$.

\begin{align}\label{eq:out_space_error_bound}
    \mathbb{E}_{m\sim \mu} \left[ \|\mathcal{F}(m)- \mathbf{F}_\theta(\Psi_r^*m)\Phi_r^*\|^2_{HS}\right]  \nonumber\\
    \leq 2\left( \mathbb{E}_{m\sim \mu} \left[ \|\mathcal{F}(m)- \mathcal{F}(m)\Phi_r\Phi_r^*\|^2_{HS}\right]  + \mathbb{E}_{m\sim \mu} \left[ \|\mathcal{F}(m)\Phi_r\Phi_r^*- \mathbf{F}_\theta(\Psi_r^*m)\Phi_r^*\|^2_{HS}\right] \right) \nonumber \\
    \leq 2\left( \sum_{j > r_q} \lambda_j  + \mathbb{E}_{m\sim \mu} \left[ \|\mathcal{F}(m)\Phi_r- \mathbf{F}_\theta(\Psi_r^*m)\|^2_{F}\right] \right)
\end{align}

As in \cite{bhattacharya2021}, the error due to the truncation can also be bounded under the assumption that $\mathcal{F}$ is Lipschitz with constant $L_\mathcal{F}$, that $\mu$ is a Gaussian measure, and $\Psi_r$ are chosen to be the first $r$ eigenfunctions of the covariance of $\mu$. This error is bounded as follows:
\begin{equation}
    \mathbb{E}_{m\sim \mu} \left[\|\mathcal{F}(m) - \mathcal{F}(\Psi_r\Psi_r^*m)\|^2_{HS}\right] \leq (L_\mathcal{F})^2 \mathbb{E}_{m\sim \mu} \left[\|m - \Psi_r\Psi_r^*m\|^2_\mathcal{M}\right]\leq (L_\mathcal{F})^2\sum_{j> r_m}^\infty \left(\sigma^\mu_j\right)^2.
\end{equation}

If we modify our operator learning task to learn $\mathcal{F}(\Psi_r\Psi_r^*m)$ instead of $\mathcal{F}(m)$, we can retrieve the following total bound:
\begin{align}\label{eq:total_space_error_bound}
    \mathbb{E}_{m\sim \mu} \left[ \|\mathcal{F}(m)- \mathbf{F}_\theta(\Psi_r^*m)\Phi_r^*\|^2_{HS}\right]  \nonumber\\
    \leq 2\left( \sum_{j > r_q}^\infty \lambda_j  + 2(L_\mathcal{F})^2\sum_{j> r_m}^\infty \left(\sigma_j^\mu\right)^2 + 2\mathbb{E}_{m\sim \mu} \left[ \|\mathcal{F}(\Psi_r\Psi_r^*m)\Phi_r- \mathbf{F}_\theta(\Psi_r^*m)\|^2_{F}\right] \right).
\end{align}
We however do not do this in practice as it is more intrusive than the former operator learning formulation.

\paragraph{Statistical sampling error, when approximating $\Phi_r$ from samples.}

In practice the solution of the generalized eigenvalue problem \eqref{eq:gevp_f} is intractable due to the inability to compute the integral $\mathbb{E}_{m\sim \mu}\left[\cdot\right]$. This problem is then approximated from i.i.d. samples $\{m_i\sim \mu \}_{i=1}^{N_\text{samples}}$ as
\begin{equation}
    \frac{1}{N_\text{samples}}\sum_{j=1}^{N_\text{samples}}\mathcal{F}(m_j)^*\mathcal{F}(m_j)\widehat{\phi_i} = \widehat{\lambda}_i \widehat{\phi_i}.
\end{equation}
From Lemma 39 in \cite{luo2025dimension}, we have the following bound on the empirical approximation from samples:
\begin{subequations}
\begin{align}\label{eq:phi_sample_error}
    \mathbb{E}_{m\sim \mu}\left[\|\mathcal{F}(m)(I_\mathcal{Q} - \widehat{\Phi}_r\widehat{\Phi}_r^*)\|^2_{HS(\mathcal{Q},\mathbb{R}^d)}\right] \leq \nonumber \\
    \sum_{j>r_q} \lambda_j + \min \left\{\sqrt{2r_q}\|\mathfrak{H} - \widehat{\mathfrak{H}}\|_{HS(\mathcal{Q},\mathcal{Q})}, \frac{2\|\mathfrak{H} - \widehat{\mathfrak{H}}\|^2_{HS(\mathcal{Q},\mathcal{Q})}}{\lambda_{r_q} - \lambda_{r_q+1}} \right\},
\end{align}
where 
\begin{align}
    \mathfrak{H} = \mathbb{E}_{m \sim \mu} \left[ \mathcal{F}(m)^* \mathcal{F}(m)\right]\\
    \widehat{\mathfrak{H}} = \frac{1}{N_\text{samples}} \sum_{j=1}^{N_\text{samples}} \mathcal{F}(m_j)^*\mathcal{F}(m_j)
\end{align}
\end{subequations}
The additional errors due to sampling can be accounted for by use of \eqref{eq:phi_sample_error}, to update the error bounds in \eqref{eq:out_space_error_bound}. Similar results can be established for the input basis.

\section{Numerical implementation}
\label{appendix:numerics}
\subsection{Example: Contaminant initial condition} \label{appendix:numerics-contaminant}

\paragraph{Velocity setup} For the wind velocity field, we use the solution of a potential flow problem modeled by the Laplace equation on domain $\Omega$ and boundary conditions given by
\begin{equation}
    \nabla^2u = 0, \qquad u = \begin{cases}
      \xi & \text{on } \partial\Omega_\text{north}\\
      0 & \text{on } \partial\Omega_\text{south}\\
    \end{cases}      
    \qquad \nabla u \cdot n  = \begin{cases}
        \zeta_1 & \text{on } \partial\Omega_\text{east} \\ 
        \zeta_2 & \text{on } \partial\Omega_\text{west} \\ 
        0 & \text{on } \partial\Omega_\text{else}
    \end{cases}
\end{equation}
where the values of $\xi,\zeta_1,\zeta_2$ are random variables which parameterize the velocity fields, and $\partial\Omega_\text{else}$ refers to all boundaries excluding the lateral boundaries $\partial\Omega_\text{north}, \partial\Omega_\text{south}, \partial\Omega_\text{east}, \partial\Omega_\text{west}$. The value of $\xi$ dictates the potential difference which dictates the predominant flow in the domain. The Neumann boundary conditions $\boldsymbol{\zeta} = [\zeta_1, \zeta_2]^\top$ induces crossflow effects. The wind velocity field is then given by $m = \nabla u$.

\paragraph{Reduced bases}

(Initial conditions) The initial conditions are sampled by randomly drawing clusters consisting of three Gaussian distributions. The centers are drawn from the bounding box defined by the points (1.0, 0.8, 0.05) and (3.0, 0.9, 0.3). Two additional centers are drawn from a uniform distribution $U(0.1, 0.2)$ for the distance from the original center. We sample 1000 initial conditions and compute the POD basis from these samples. The eigenvalue decay is shown on the left in Figure~\ref{fig:contaminant-sv-decay}. We choose $r_q = 50$ for the reduced dimension. 

(Wind velocities) We sample the uniform distribution $\xi \sim U(0.1, 0.5)$. This results in a southerly dominant flow direction, which is the most common wind direction based on historical wind data~\cite{winddata}. For the Neumann conditions, we sample the normal distribution $ \boldsymbol{\zeta} \sim N(\mathbf{0}, \mathbf{\Gamma})$, where $\mathbf{\Gamma} = \begin{bmatrix} 0.0025 & -0.0001 \\ -0.0001 & 0.0025\end{bmatrix}$. These values, along with the diffusion coefficient $\kappa = 1\mathrm{e}{-4}$ and $T=12$, results in feasible P\'{e}clet number ranges in the advection-diffusion problem for atmospheric conditions. We draw 500 independent samples of the parameters, and solve the potential flow problem to obtain the wind velocity field. From these fields, we compute the POD basis. The eigenvalue decay is shown on the right in Figure~\ref{fig:contaminant-sv-decay}. We observe a large drop in the decay after three modes, and therefore choose $r_m = 3$ for the reduced dimension. 

\begin{figure}[h!]
    \centering
    \includegraphics[width=0.75\linewidth]{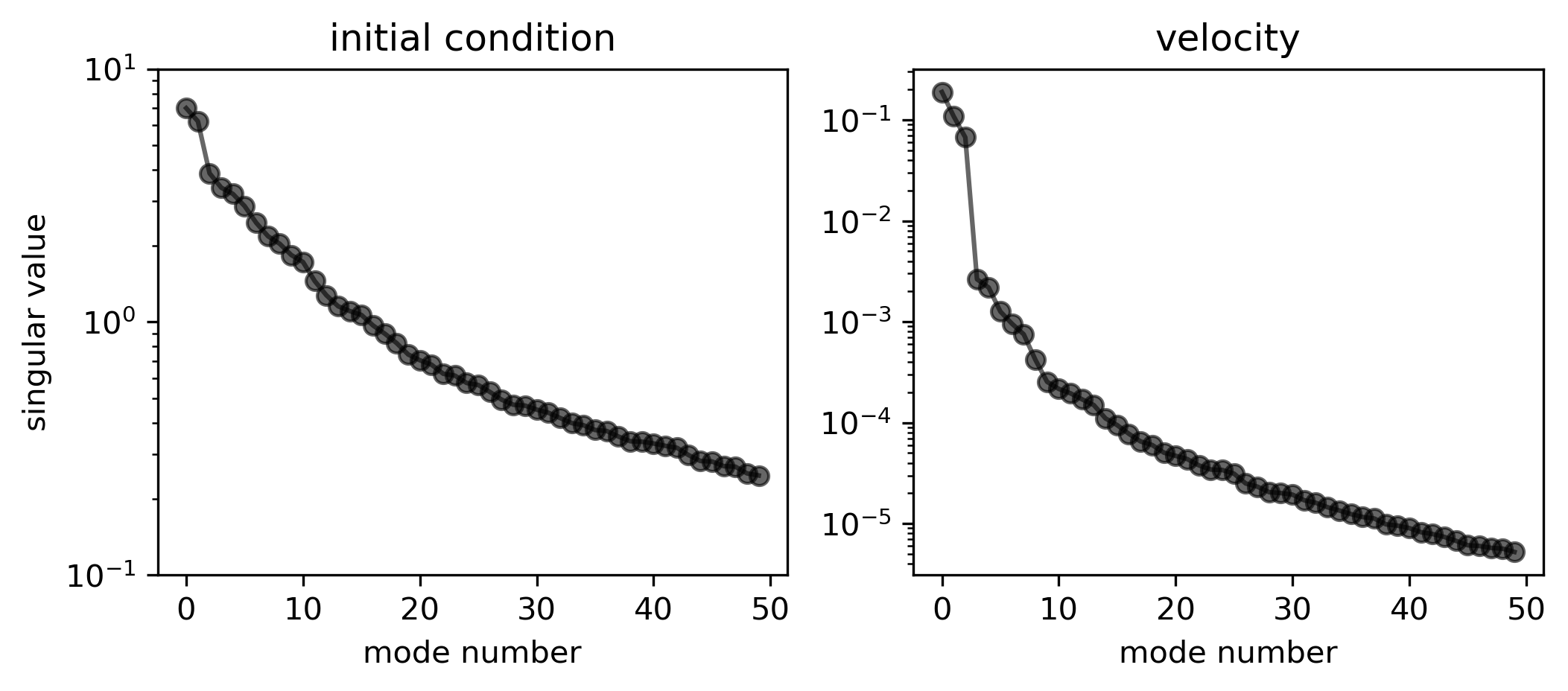}
    \caption{Eigenvalue decay for initial conditions (inversion parameters), and wind velocity fields (model parameters).}
    \label{fig:contaminant-sv-decay}
\end{figure}

\subsection{Comparison}\label{appendix:comparison}

\subsubsection{Forward surrogate}\label{appendix:comparison-forward}

The forward comparison model $\mathcal{G}_\theta$ is a multiple-input neural operator (MIONet) with two branch networks and a finite dimensional output. The inputs to each branch network are the reduced coordinate representations of the inversion and model parameters. The outputs of each branch network have equal dimensions (the latent dimension) and are aggregated via the Hadamard product. For the trunk network, a linear layer is used as a basis expansion for the finite dimensional observations, analogous to the case in Appendix B in~\cite{jin2022mionet}. For the inverse problem, the solution using $\mathcal{G}_\theta$ is given by
\begin{equation}
    \mathbf{q}_r^\star = \argmin_{\mathbf{q}_r} \| \mathcal{G}_\theta(\mathbf{q}_r, \mathbf{m}_r) - \mathbf{d} \|_2^2 + \gamma\|\mathbf{q}_r\|_2^2
\end{equation}
This is an equivalent subspace inverse problem formulation as is solved by the normal equations in Eq.~\eqref{eq:subspace_ip_nemo}. For the online optimization of this inverse problem, we used the L-BFGS optimizer in PyTorch with a gradient tolerance of $10^{-3}$ for termination, with 100 maximum iterations. 

For training $\mathcal{G}_\theta$, we use the same number of model parameter samples as is used to train NEMO in each example application. However, here we also require samples of the inversion parameters. For the contaminant problem, we compute 2000 samples of the initial conditions, and for the hypersonic problem, we compute 1000 samples of pressure fields. All combinations of inversion and model parameters are used to compute the observational data, which is then split into training and validation sets for training $\mathcal{G}_\theta$. The model is trained to obtain comparable output performance accuracy as NEMO, and the architecture with the fewest number of parameters to achieve this accuracy is selected through a hyperparameter grid search. The number of hidden layers is kept equal for each of the branch networks. The activation function used is the same as used in the corresponding NEMO model. 

In the contaminant problem, we consider a hidden layer depth 2 to 4, $q$-branch width range of 100 to 300, $m$-branch width range of 10 to 200, and latent size of 300 to 400. The selected model consisted of 2 hidden layers and width of 100 for the branch networks, and a shared latent dimension of 400. In the hypersonic problem, we consider a hidden layer depth 1 to 3, $q$- and $m-$branch width range of 50 to 200, and latent size of 200 to 300. The selected model consisted of 2 hidden layers, width of 200 for the $q$-branch network and 100 for the $m$-branch network, and a latent dimension of 300.  

\subsubsection{Direct inverse map} \label{appendix:comparison-direct}

The direct inverse map also uses the MIONet architecture, however the branch network inputs are the observations and the reduced coordinate representation of the model parameters, while the output is the reduced coordinate representation of the inversion parameters. In this work, we take the approach of training the direct inverse map on the true reduced inverse problem solution, $\mathbf{q}_r^\star$, which is obtained via the normal equations given in Equation~\ref{eq:subspace_ip}. The inverse maps are trained with the number of training data points equivalent to the number of forward solves required to produce the NEMO training data, i.e. $r_q$ times the number of samples of $m$ used for training NEMO. 

In the contaminant problem, the inverse MIONet model consisted of 2 hidden layers with a shared latent dimension of 500 with a $d$-branch width of 300 and $m$-branch width of 100. The model is trained for 1000 epochs using the Adam optimizer with an initial learning rate of $1\times 10^{-3}$ and batch size 64, and utilizing \texttt{ReduceLROnPlateau} of PyTorch with decay factor of $0.9$ and patience $50$.
For the hypersonic problem, the inverse MIONet model consisted of 3 hidden layers with a shared latent dimension of 500 with a $d$-branch width of 400 and $m$-branch width of 200. The model is trained for 2000 epochs using the Adam optimizer with an initial learning rate of $1\times 10^{-4}$ and batch size 64, and utilizing \texttt{ReduceLROnPlateau} of PyTorch with decay factor of $0.9$ and patience $30$.

\end{document}